\documentclass[oneside,english]{amsart}
\usepackage[T1]{fontenc}
\usepackage[latin9]{inputenc}
\usepackage{geometry}
\usepackage{mathrsfs}
\usepackage{amsthm}
\usepackage{amssymb}
\usepackage{esint}
\newcommand{\ddd}{\vec}
\newcommand{\rrr}{\hat}
\newtheorem{rem}{Remark}[section]
\newtheorem{rem*}{Remark}[section]
\newtheorem{assmpt}{Assumption}[section]
\newtheorem{assmpt*}{Assumption}[section]

\makeatletter

\providecommand{\tabularnewline}{\\}

\numberwithin{equation}{section}
\numberwithin{figure}{section}
\theoremstyle{plain}
\newtheorem{thm}{\protect\theoremname}
  \theoremstyle{plain}
  \newtheorem{cor}[thm]{\protect\corollaryname}
  \theoremstyle{plain}
  \newtheorem{lem}[thm]{\protect\lemmaname}
 \theoremstyle{definition}
 \newtheorem*{defn*}{\protect\definitionname}

\usepackage{libertine}

\makeatother

\usepackage{babel}
  \providecommand{\corollaryname}{Corollary}
  \providecommand{\definitionname}{Definition}
  \providecommand{\lemmaname}{Lemma}
  
\providecommand{\theoremname}{Theorem}

\begin{document}

\title{Convergence Rates for Hierarchical Gibbs Samplers}

\author{Oliver Jovanovski{*}}

\author{Neal Madras}

\maketitle
\emph{Department of Mathematics and Statistics, York University, 4700
Keele Street, Toronto, Ontario M3J 1P3}
\begin{abstract}
We establish results for the rate of convergence in total variation 
of a Gibbs sampler to its equilibrium distribution. This sampler is
motivated by a hierarchical Bayesian inference construction for a
gamma random variable. Our results apply to a wide range of parameter
values for the case that the hierarchical depth is 3 or 4. Our method involves showing
a relationship between the total variation of two ordered copies of
our chain and the maximum of the ratios of their respective coordinates.
We construct auxiliary stochastic processes to show that this ratio
does converge to 1 at a geometric rate.
\end{abstract}
\bigskip{}

\begin{tabular}{cl}
\emph{Key words and phrases:} & Convergence Rate, Hierarchical Gibbs Sampler, Markov Chain, Coupling, \tabularnewline
 & Gamma Distribution, Stochastic Monotonicity\tabularnewline
\end{tabular}

\bigskip{}

\section{Introduction}

A basic purpose of Markov chain Monte Carlo (MCMC) is to generate samples 
from a given ``target'' probability distribution by inventing a Markov chain that 
has the target as its equilibrium, and then sampling from long runs of this chain.
There is a significant amount of theory showing that a Markov chain
satisfying some fairly general conditions (see for example \cite{key-1})
will converge to an equilibrium in distribution, as well as  in the
stronger measure of total variation. Mere knowledge of convergence
is often not enough, and it is of both theoretical and practical
interest to consider the rate at which convergence proceeds. In particular,
deriving an upper bound on the rate of convergence would provide a
rigorous degree of certainty as to how far this Markov chain
is from its equilibrium distribution, and it would help assess
the efficiency of this sampling procedure.

This has been our main objective in this paper, where the model in
question is motivated by the following hierarchical Bayesian inference
scenario.  We are given a real number $x>0$ with the information that
it was drawn from a $\Gamma\left(a_{1},u_{1}\right)$ distribution,
i.e.\ the Gamma distribution with probability density function 
\[
f\left(z\right)=\frac{u_{1}^{a_{1}}}{\Gamma\left(a_{1}\right)}z^{a_{1}-1}e^{-zu_{1}}
\hspace{5mm}(z>0)\,.
\]
Here the shape parameter $a_{1}>0$ is fixed, but the inverse scale parameter $u_{1}$ is
itself the product of random sampling from an independent $\Gamma\left(a_{2},u_{2}\right)$
distribution. Once again we assume that $a_{2}>0$ is a given
constant, while $u_{2}$ is sampled in an analogous manner. This process
continues until we reach $u_{n}\sim\Gamma\left(a_{n+1},b\right)$,
where now both $a_{n+1}>0$ and $b>0$ are given. The joint
density of $\left(x,u_{1},\ldots,u_{n}\right)$ is therefore proportional to 
\begin{equation}
  p\left(x,z_{1},\ldots,z_{n}\right)\propto x^{a_{1}-1}\left(\prod_{i=1}^{n}z_{i}^{a_{i}+a_{i+1}-1}\right)
  \exp\left(\sum_{i=1}^{n+1}-z_{i}z_{i-1}\right)\label{eq:jdensity}
\end{equation}
where for convenience we set $z_{0}:=x$ and $z_{n+1}:=b$. We conclude
from (\ref{eq:jdensity}) that for $1\leq i\leq n$, the conditional
distribution of $u_{i}$ given everything else is 
\[
u_{i}\left|x,u_{j\neq i}\right.\sim\Gamma\left(a_{i}+a_{i+1},u_{i-1}+u_{i+1}\right)
\]
Therefore the resulting posterior distribution of $u=\left(u_{1},\ldots,u_{n}\right)$
(i.e. given $x$ as well as all other parameters) has the density function 
\begin{equation}
  g(z_{1},\ldots,z_{n})\propto\left(\prod_{i=1}^{n}z_{i}^{a_{i}+a_{i+1}-1}\right)
   \exp\left(\sum_{i=1}^{n+1}-z_{i}z_{i-1}\right)\label{eq:equilibrumdensity}
\end{equation}

Such Bayesian inference networks have been a popular statistical representations
used to handle problems ranging from sports predictions and gambling
to genetics, disease outbreak detection and artificial intelligence
(\cite{key-11}, \cite{key-12}, \cite{key-13}, \cite{key-14}) .
Gibbs samplers are frequently associated with problems in Bayesian
inference, which is also the case for the problem considered in this paper. 
Similar constructions have appeared in numerous
statistical models used in a variety of applications. In information
retrieval related to search engines, hierarchical models are used
to decide how to represent documents based on relevant queries (see
for example \cite{key-15}). Multi-Population Haplotype Phasing is
a problem in statistical genetics where hierarchical Bayesian models
can be used to represent genotypes ( e.g. \cite{key-16}), and in
market research similar models are used in predicting buyer behaviour
and decision making (\cite{key-17}). 

The Gibbs sampler \cite{key-4-1} has been a very popular MCMC algorithm
for obtaining a sample from a probability distribution that is 
difficult to sample from directly. In its fundamental form, this algorithm works
on a vector $u$ by selecting (systematically, randomly or otherwise)
one of the vector's components $u_{i}$ and updating this component
only, by drawing from the probability distribution of $u_{i}$ given
$\left(u_{j\neq i}\right)$. 

General convergence results have been derived for some Gibbs samplers
(e.g. \cite{key-4-2}), however due to their limitations it is often
not possible to infer quantitative bounds directly from these results.

In this paper we will focus on the case when $n=4$, with a short section dedicated to the immediate results that follow for the case $n=3$. For values of $n>4$ we refer the reader to \cite{key-18}, where we derive similar results under stricter constraints on the parameters.

\subsection{The problem}

	Our aim is to construct a
Gibbs sampler on $\mathbb{R}_{+}^{4}$ and show that it converges rapidly to the target
distribution with density function given by (\ref{eq:equilibrumdensity}) with $n=4$. For $n>4$, we give a similar approach in \cite{key-18}. 

\underline{Notation}:  We shall write $\ddd{u} = (u_1,u_2,u_3,u_4)$ for points in $\mathbb{R}^{4}$.
We shall often refer to points of $\mathbb{R}^2$ consisting of the second and fourth entries of
$\ddd{u}$.  We shall then omit the $\ddd{\quad}$ and write $u=(u_2,u_4)$.

We first consider the Markov chain which sequentially updates
its coordinates as follows. For $i\in\left\{ 1,2,3,4 \right\}$, let 
\[
   \bar{P}_{i}\left(\ddd{v},d\ddd{w}\right):=\left(\prod_{j\neq i}\delta_{v_{j}}(w_{j})\right)
   h_{i}(w_{i}\left|\ddd{v}\right.)\,dw_{i}
\]
where $h_{i}(\cdot\left|\ddd{v}\right.)$ is the $\Gamma(a_i+a_{i+1},v_{i-1}+v_{i+1})$ density function 
given $\ddd{v}$, and where for convenience we have defined $v_{0}:=x$
and $v_{5}:=b$. In other words, $\bar{P}_{i}$ is the probability kernel that updates (only)
the $i^{th}$ coordinate according to the conditional density $h_{i}$.  Now define
\begin{equation}
\bar{P}:=\bar{P}_{1}\bar{P}_{3}\bar{P}_{2}\bar{P}_{4},  \label{eq:pbar}
\end{equation}
the Gibbs sampler Markov chain that updates the odd coordinates and then the even coordinates.
We will show that $\bar{P}$  converges to equilibrium at a geometric rate, and we will give a bound on the rate of convergence. 

To describe distance from equilibrium, we use the total variation metric $d_{TV}$, 
which is defined as follows.  For two probability measures $\mu_1$ and $\mu_2$ on
the same state space $\Omega$, define 
$d_{TV}(\mu_1,\mu_2) \,:=\,\inf\,\mathbb{P}(X_1\neq X_2)$,
where the infimum is over all joint distributions $\mathbb{P}$ of $(X_1,X_2)$ such that 
$X_1\sim \mu_1$ and $X_2\sim\mu_2$.  
If $Y_i$ denotes a random variable with distribution $\mu_i$, then we shall also write 
$d_{TV}(Y_1,Y_2)$
and $d_{TV}(Y_1,\mu_2)$ for $d_{TV}(\mu_1,\mu_2)$.
It is known (e.g.\ Chapter I of \cite{key-2}) that the infimum 
is achieved by some $\mathbb{P}$,
and that we can also express  $d_{TV}(\mu_1,\mu_2)$ as the supremum of 
$|\mu_1(A)-\mu_2(A)|$ over all measurable $A\subset \Omega$.  

It will be useful to represent our Markov chain using iterated random functions (\cite{DiaFre},
\cite{MadSez})
as follows.  Let $\{\gamma_i^t: \,t=1,2,\ldots, i=1,2,3,4\}$ be a collection of independent
random variables with each $\gamma_i^t$ having the $\Gamma(a_i+a_{i+1},1)$ distribution.  
Then define the  sequence of
random functions $\bar{F}^t:\mathbb{R}_+^4\rightarrow\mathbb{R}_+^4$ ($t=1,2,\ldots$) by
\begin{eqnarray}
    \nonumber
    \bar{F}^t(\ddd{u})  &= & (\bar{F}^t_1(\ddd{u}),\,\bar{F}^t_2(\ddd{u}),\,\bar{F}^t_3(\ddd{u}),\,
    \bar{F}^t_4(\ddd{u}))   \\
       \label{eq:mc4D}
       & = & \left(   \frac{\gamma^t_1}{x+u_2}, \;\frac{\gamma^t_2}{\frac{\gamma^t_1}{x+u_2}+
        \frac{\gamma^t_3}{u_2+u_4}}, \; \frac{\gamma^t_3}{u_2+u_4},\;
        \frac{\gamma^t_4}{b+ \frac{\gamma^t_3}{u_2+u_4}} \right) .
\end{eqnarray}
Then for any initial $\ddd{u}^0\in \mathbb{R}_+^4$, the random sequence 
$\ddd{u}^0,\ddd{u}^1,\ldots$ defined recursively by $\ddd{u}^{t+1}=\bar{F}^{t+1}(\ddd{u}^t)$ 
is a Markov chain with transition kernel $\bar{P}$.

Observe that $\bar{F}^t(\ddd{u})$ does not depend on $u_1$ or $u_3$.
It follows that if $\{u^t\}$ is a version of the Markov chain (\ref{eq:pbar}), then the sequence
$\{\left(u_{2}^{t},u_{4}^{t}\right)\}$ is itself a Markov chain in $\mathbb{R}^2_+$. 
Accordingly, we define the  random functions 
$F^t:\mathbb{R}_+^2\rightarrow\mathbb{R}_+^2$  ($t=1,2,\ldots$) by
\[
    F^t(u_2,u_4) \;=\; \left( F_2^t(u_2,u_4),F_4^t(u_2,u_4)\right)  \;=\;
     \left(\frac{\gamma^t_2}{\frac{\gamma^t_1}{x+u_2}+ \frac{\gamma^t_3}{u_2+u_4}}, \;
        \frac{\gamma^t_4}{b+ \frac{\gamma^t_3}{u_2+u_4}} \right) .
\]
Thus $\bar{F}^t_i(\ddd{u})=F^t_i(u_2,u_4)$ for $i=2,4$ and all $\ddd{u}\in \mathbb{R}_+^4$ 
and all $t$.
Moreover, the Markov chain $\{\left(u_{2}^{t},u_{4}^{t}\right)\}$ is given by the random 
recursion 
\begin{equation}
     (u_{2}^{t+1},u_{4}^{t+1})\,=\,F^{t+1}(u^t_2,u^t_4).
 \label{eq:mc2}
\end{equation}

Let $\bar{\pi}$ be the probability measure on $\mathbb{R}_{+}^{4}$
with density function (\ref{eq:equilibrumdensity}).  Then it is well
known (e.g. see Section 2.3 of \cite{key-9}) that $\bar{\pi}$
is the equilibrium distribution of the Markov chain defined by (\ref{eq:pbar}).
It follows that the marginal distribution of the even coordinates
of $\bar{\pi}$, which we denote by $\pi$, is the equilibrium distribution
of (\ref{eq:mc2}). Furthermore, the following simple argument illustrates that it suffices to 
bound the distance to equilibrium of  (\ref{eq:mc2}).

\begin{lem}
\label{lem:s4relate}Let $\ddd{\Upsilon}^{t}$ be a copy of the Markov chain
(\ref{eq:pbar}) on $\mathbb{R}_{+}^{4}$ and let $\Phi^{t}$
be a copy of (\ref{eq:mc2}) on $\mathbb{R}_{+}^{2}$.  
Assume $(\Upsilon^0_2,\Upsilon^0_4)\,=\,\Phi^0$, i.e.\ the initial conditions agree.
Then $d_{TV}\left(\ddd{\Upsilon}^{t+1},\bar{\pi}\right)\leq d_{TV}\left(\Phi^{t},\pi\right)$.
\end{lem}
\begin{proof}
There exists a jointly distributed pair of random vectors $(\Psi,\Lambda)$ with
 $\Psi=\left(\Psi_{2},\Psi_{4}\right) \sim\Phi^{t}$
and $\Lambda=\left(\Lambda_{2},\Lambda_{4}\right)\sim\pi$
such that 
$
d_{TV}\left(\Phi^{t},\pi\right)=\mathbb{P}\left[\Psi\neq\Lambda\right] \,.
$
Then $\ddd{\Upsilon}^{t+1}\sim \bar{F}^{t+1}(1,\Psi_2,1,\Psi_4)$ and 
$\bar{F}^{t+1}(1,\Lambda_2,1,\Lambda_4)\sim \bar{\pi}$.
Hence
\begin{eqnarray*}
d_{TV}\left(\ddd{\Upsilon}^{t+1},\bar{\pi}\right)  & \leq & 
  \mathbb{P}\left[\bar{F}^{t+1}(1,\Psi_2,1,\Psi_4)\neq \bar{F}^{t+1}(1,\Lambda_2,1,\Lambda_4)\right]
  \\
 & \leq & \mathbb{P}\left[\Psi\neq\Lambda\right]\\
 & = & d_{TV}\left(\Phi^{t},\pi\right) \,.
\end{eqnarray*}
\end{proof}

We can now state our main results. Let $\mathcal{U}^{t}$
and $\mathcal{W}^{t}$ be two copies of the Markov chain (\ref{eq:mc2}) starting
at points $\mathcal{U}^{0}$ and $\mathcal{W}^{0}$ respectively, and let $d_{TV}$ denote the total variation metric. 
We define the condition 
\begin{equation}
a_{1}+a_{4}>1,a_{2}+a_{5}>1,a_{2}+a_{3}>1,a_{3}+a_{4}>1,a_{4}+a_{5}>1 \,. \label{eq:condition1}
\end{equation}
Let  $M:=\max_{i}\left\{ \mathcal{U}_{i}^{0},\mathcal{W}_{i}^{0}\right\} $, 
$m=\min_{i}\left\{ \mathcal{U}_{i}^{0},\mathcal{W}_{i}^{0}\right\} $ for $i\in {2,4}$ and define
$R_{0}=\frac{M}{m}$, $J_{0}=2m+1/\left(2m\right)$. The constants $\eta$, $d$ and $r$  appearing in the statement of Theorem \ref{thm:thm1} are defined in Appendix A, and depend only on the parameters $x,b,a_{1},\ldots,a_{5}$. 

\begin{thm}
\label{thm:thm1} Assume that (\ref{eq:condition1}) holds, and fix $\mathcal{U}^0$ and 
$\mathcal{W}^0$. 
If $J_{0}\leq\eta$, then for $t>0$,
\[
   d_{TV}\left(\mathcal{U}^{t+3},\mathcal{W}^{t+3}\right) \;  \leq \; 3\, r^{t/2d}
  \left(a_{2}+a_{3}+a_{4}+a_{5}\right)\left(R_{0}-1\right) \,.
\]
For general values of $J_{0}$, we have
\[
d_{TV}\left(\mathcal{U}^{t+3},\mathcal{W}^{t+3}\right) \;\leq \; 3\,r^{t/4d}\left(a_{2}+a_{3}+a_{4}+a_{5}\right)\left(R_{0}-1\right)+\frac{\max\left\{ J_{0},\eta\right\} }{\eta}\beta^{\left\lfloor \frac{t}{2}\right\rfloor +3} .
\]
\end{thm}

\noindent
We explain our need for condition (\ref{eq:condition1}) in section \ref{sub:ExistenceDt}. 
Taking $\mathcal{W}^{0}{\sim}\pi$ leads to the following.
\begin{cor}
\label{cor:cor2} Assume that (\ref{eq:condition1}) holds, and fix $\mathcal{U}^0$. For $t>0$,
\[ 
  d_{TV}\left(\mathcal{U}^{t+3},\pi\right) \;\leq\; 3\, r^{t/4d}\left(a_{2}+a_{3}+a_{4}+a_{5}\right)
 \mathbb{E}_{\pi}\left[R_{0}-1\right] +\left(\frac{\mathbb{E}_{\pi}\left[J_{0}\right]}{\eta}+1\right)\beta^{\left\lfloor \frac{t}{2}\right\rfloor +3} .
\]
\end{cor}
The quantities $\mathbb{E}_{\pi}\left[R_{0}\right]$ and $\mathbb{E}_{\pi}\left[J_{0}\right]$
depend only on $\pi$ and $\mathcal{U}^0$, and can be estimated with a bit of effort.
This is done in  Appendix B of \cite{key-18} for the case $\mathcal{U}^{0}=(1,1)$
(see also the end of Section \ref{sub:Sampling} in the present paper).

\subsection{Outline of our proof}
Essentially the proof of Theorem \ref{thm:thm1} relies on a coupling argument. In Section
\ref{sub:Stochastically-monotone-coupling} we consider a partial order
``$\preceq$'' on $\mathbb{R}_{+}^{2}$
and show that we can couple two copies $\left\{ u^{t},w^{t}\right\} $
of (\ref{eq:mc2}) with the initial condition $u^{0}\preceq w^{0}$
in a monotone manner, thus preserving the order $u^{t}\preceq w^{t}$,
up to a ``one-shot coupling'' time at which $u^t$ and $w^t$ try to coalesce
(succeeding with high probability, desirably).
In the beginning of Section \ref{sec:RatioRt}
we show that if $R_{t}$ is a process that serves as an upper bound
for the ratio $\underset{i}{\max}\left\{ \frac{w_{i}^{t}}{u_{i}^{t}}\right\} $,
then the rate of convergence of $R_{t}\rightarrow1$ can be related
to the rate at which (\ref{eq:mc2}) converges to equilibrium. Therefore,
our focus becomes the question of how to define such a process and show that
it converges to $1$ at a geometric rate. 
In Section \ref{sec:RatioRt} we define a 
stochastic process adapted to the same filtration as $u^{t}$,
with the property that it is an upper bound of (in the sense of $\preceq$) a 
copy of (\ref{eq:mc2}) started at $w^{0}$. This allows
us to define $R_{t}$ which has the additional quality of being strictly
monotone decreasing. This alone does not guarantee that $R_{t}\rightarrow1$
quickly (or at any pace, for that matter). But the rate at which $R_{t}$
approaches $1$ does depend on the size of the values $u_{2}^{t}$
and $u_{4}^{t}$, and we show that if often enough these two values
are neither too large nor too small, then $R_{t}\rightarrow1$ at
a geometric rate. To fulfill this condition, we postulate a number of
auxiliary processes in Section \ref{sub:Super-martingale} (and construct
them in Section \ref{sub:ExistenceDt}) that provide
upper bounds for the terms $\left\{ u_{2}^{t},u_{4}^{t},\frac{1}{u_{2}^{t}},\frac{1}{u_{4}^{t}}\right\} $,
and we show that they are frequently less than a fixed constant $\eta$. 
\medskip{}

\section{\label{sub:Stochastically-monotone-coupling}Monotone
coupling and one-shot coupling}

For $u=\left(u_{2},u_{4}\right)\in\mathbb{R}_{+}^{2}$ and 
$w=\left(w_{2},w_{4}\right)\in\mathbb{R}_{+}^{2}$, define the partial order
$u\preceq w$ to mean $u_2\leq w_2$ and $u_4\leq w_4$.
Given two initial points $u^0$ and $w^0$, we can produce two versions of the Markov chain
(\ref{eq:mc2}) in $\mathbb{R}^2_+$ using $u^{t+1}=F^{t+1}(u^t)$ and $w^{t+1}=F^{t+1}(w^t)$
(crucially, we use the same random variables $\left\{ \gamma_{i}^{t}\right\} $
in both versions).  We refer to this as the ``uniform coupling.''
This coupling is monotone, in the sense that if $u^0\preceq w^0$ 
then $u^{t}\preceq w^{t}$ for all times $t$. Suppose
we couple two copies in this manner  commencing at arbitrary
initial points $\mathcal{U}^{0},\mathcal{W}^{0}\in\mathbb{R}_{+}^{2}$.   Then
we can take $m=\min\left\{ \mathcal{U}_{2}^{0},\mathcal{U}_{4}^{0},\mathcal{W}_{2}^{0},
\mathcal{W}_{4}^{0}\right\}$
and $M=\max\left\{ \mathcal{U}_{2}^{0},\mathcal{U}_{4}^{0},\mathcal{W}_{2}^{0},\mathcal{W}_{4}^{0}\right\}$,
and define 
\begin{equation}
   w^{0}:=(M,M)\in\mathbb{R}_{+}^{2} \hspace{5mm}  \hbox{and}  \hspace{5mm}
  u^{0}:=(m,m)\in\mathbb{R}_{+}^{2} \,.
\label{eq:setvu}
\end{equation}
Observing that $u^{0}\preceq\{\mathcal{U}^{0},\mathcal{W}^{0}\}\preceq w^{0}$,
we conclude that $\mathcal{U}^{t}$ and $\mathcal{W}^{t}$ are perpetually
``squeezed'' between $u^{t}$ and $w^{t}$ (i.e.,
$u^{t}\preceq\{\mathcal{U}^{t},\mathcal{W}^{t}\}\preceq w^{t}$ for all $t$).
Corollary \ref{cor:corollary2} below justifies why it suffices
to consider the coupled pair $\left(u^{t},w^{t}\right)$ in order
to bound $d_{TV}\left(\mathcal{U}^{t},\mathcal{W}^{t}\right)$.

\begin{lem}
\label{lem:Lemma 1}Suppose that $0<\beta_{1}<\beta_{2}<\beta_{3}<\beta_{4}$ and $\alpha>0$.
Let $f_i$ be the density function of  $Z_i\sim \Gamma\left(\alpha,\beta_{i}\right)$. Then 
\[
   \min\left\{ f_{1}(y),f_{4}(y)\right\} \leq \min\left\{ f_{2}(y),f_{3}(y)\right\} \hspace{5mm}
    \hbox{for all $y\geq 0$}.
\]
\end{lem}
\begin{rem} Since a property of total variation (see Proposition 3 of \cite{key-1}) is that
\[
d_{TV}\left(Z_{i},Z_{j}\right)\;=\;1-\int \min\left\{f_{i}\left(y\right),f_{j}\left(y\right)\right\}dy\,,
\]
we also conclude from 
Lemma \ref{lem:Lemma 1} that 
$d_{TV}\left(Z_{2},Z_{3}\right)\,\leq \,d_{TV}\left(Z_{1},Z_{4}\right)$.
\end{rem}
\begin{proof}[Proof of Lemma \ref{lem:Lemma 1}]
Note first that for $i,j\in\left\{ 1,2,3,4\right\}$ with $i<j$, 
\begin{equation}
   f_{i}(y) \; \geq \; f_{j}(y) 
   \;\; \Longleftrightarrow \;\; 
    \beta_{i}^{\alpha}\exp(-\beta_{i}y) \; \geq\; \beta_{j}^{\alpha}\exp(-\beta_{j}y)
    \;\; \Longleftrightarrow \;\; 
    y \;\geq \; g(\beta_i,\beta_j) \,
     \label{eq:lemma121eq1}
\end{equation}
where
\[
g(\beta_i,\beta_j):=\frac{\alpha\left(\ln(\beta_i)-\ln(\beta_j)\right)}{\beta_i-\beta_j} \,.
\]
Observe that $g(\beta_i,\beta_j)$ is the slope of the secant line joining the points
$(\beta_i,z_i)$ and $(\beta_j,z_j)$ on the curve $z=\alpha\ln\beta$.  Since this curve is 
concave, the slope is decreasing in $\beta_i$ and $\beta_j$ for $\beta_i<\beta_j$
(Lemma 5.16 of \cite{key-Roy}), which implies
\begin{equation}
  g(\beta_{4},\beta_{3}) \,\leq\, g(\beta_{4},\beta_{2}) \,\leq\, g(\beta_{4},\beta_{1})\,=\,
   g(\beta_{1},\beta_{4}) \,\leq\, g(\beta_{1},\beta_{3})\,\leq\, g(\beta_{1},\beta_{2}) \,.
  \label{eq:lemma121eq2}
\end{equation}
Then from (\ref{eq:lemma121eq1}) and (\ref{eq:lemma121eq2}) it follows
that 
\begin{eqnarray*}
f_{1}(y) & \leq & \min\left\{ f_{2}(y),f_{3}(y)\right\} \: \hbox{ on }\:[0,g(\beta_{1},\beta_{3})] \;
    \hbox{ and }\\
f_{4}(y) & \leq & \min\left\{ f_{2}(y),f_{3}(y)\right\} \: \hbox{ on }\:[g(\beta_{4},\beta_{2}),\infty) \,;
\end{eqnarray*}
hence 
$
 \min\left\{f_{1}(y),f_{4}(y)\right\}\leq \min\left\{ f_{2}(y),f_{3}(y)\right\} \: \hbox{ on }
   \:[0,g(\beta_{1},\beta_{3})]\cup[g(\beta_{4},\beta_{2}),\infty)=[0,\infty) \,.
$
\end{proof}

We now describe ``one-shot coupling'' of the Markov chains $u$, $w$, $\mathcal{U}$, and
$\mathcal{W}$ at time $t{+}1$ (described in \cite{key-10} in greater generality).
Assume that the uniform coupling of these chains hold up to and including time $t$.
The two random variables $\gamma_{1}^{t+1}$ and $\gamma_{3}^{t+1}$ will be used
for all four chains at time $t{+}1$.
%
For $i\in\{2,4\}$,
let $f_{u_i}$ be the probability density function of the conditional distribution of $u_i^{t+1}$
given $u^t$ and $\left(\gamma_{1}^{t+1},\gamma_{3}^{t+1}\right)$,  with analogous definitions for
$f_{w_i}$, $f_{\mathcal{U}_{i}}$, and $f_{\mathcal{W}_{i}}$.
For each coordinate $i\in\{2,4\}$, we take $u_{i}^{[t+1]C}$
to be the $x$-coordinate of a uniformly chosen point from the area under
the graph of the density function $f_{u_{i}}$. 
(The superscript $[t+1]C$ denotes that the coupling occurs at time $t+1$.)
If this point also lies below the graph of the density function $f_{w_{i}}$,
then set $w_{i}^{[t+1]C}=\mathcal{W}_{i}^{[t+1]C}=\mathcal{U}_{i}^{[t+1]C}=u_{i}^{[t+1]C}$.
Otherwise, let $w_{i}^{[t+1]C}$
be the $x$-coordinate of a uniformly and independently chosen point
from the area above the graph of $\min\left\{ f_{u_{i}},f_{w_{i}}\right\} $
and below the graph of $f_{w_{i}}$ (in this case, $w^{[t+1]C}\neq u^{[t+1]C}$ because
$f_{u_i}(u^{[t+1]C})>f_{w_i}(u^{[t+1]C})$),  
let $\mathcal{W}_{i}^{[t+1]C}$
be the $x$-coordinate of a uniformly and independently chosen point
from the area above the graph of $\min\left\{ f_{u_{i}},f_{v_{i}}\right\} $
and below the graph of $f_{\mathcal{W}_{i}}$, and let  $\mathcal{U}_{i}^{[t+1]C}$
be the $x$-coordinate of a uniformly and independently chosen point
from the area above the graph of $\min\left\{ f_{u_{i}},f_{w_{i}}\right\} $
and below the graph of $f_{\mathcal{U}_{i}}$. By Lemma
\ref{lem:Lemma 1} we know $\min\left\{ f_{u_{i}},f_{w_{i}}\right\} \leq \min\left\{ f_{\mathcal{W}_{i}},f_{\mathcal{U}_{i}}\right\} $, hence it is easy to verify that 
$(\mathcal{U}^{[t+1]C},\mathcal{W}^{[t+1]C},u^{[t+1]C},w^{[t+1]C})$
is indeed a coupling of $\left(\mathcal{U}^{t+1},\mathcal{W}^{t+1},u^{t+1},w^{t+1}\right)$.   
(Observe that the relations $u^{[t+1]C} \preceq \{\mathcal{U}^{[t+1]C},\mathcal{W}^{[t+1]C}\}
\preceq w^{[t+1]C}$ may not hold.)
\begin{cor}
\label{cor:corollary2} For one-shot coupling at time $t+1$, we have 
\[
    d_{TV}\left(\mathcal{U}^{t+1},\mathcal{W}^{t+1}\right)\;\leq\;
   \mathbb{P}\left[u^{[t+1]C}\neq w^{[t+1]C}\right]  \,.
\]
\end{cor}
\begin{proof}
By the coupling construction, 
$\left\{ \mathcal{U}_{i}^{[t+1]C}\neq\mathcal{W}_{i}^{[t+1]C}\right\} \subseteq\left\{ u_{i}^{[t+1]C}
\neq w_{i}^{[t+1]C}\right\}$ for $i=2,4$.   Therefore
\[   
   d_{TV}\left(\mathcal{U}^{t+1},\mathcal{W}^{t+1}\right) \;\; \leq \;\; 
  \mathbb{P}\left[\mathcal{U}^{[t+1]C}\neq\mathcal{W}^{[t+1]C}\right]  \;\;\leq \;\;
  \mathbb{P}\left[{u}^{[t+1]C}\neq {w}^{[t+1]C}\right] \,.  
\]
\end{proof}

\section{\label{sec:RatioRt}The ratio $R_{t}$} 

We assume in this section that the chains $u$ and $w$ are constructed by the uniform
coupling with $u^0\preceq w^0$ (which holds by (\ref{eq:setvu})), so that 
$u^{t}\preceq w^{t}$ and  $w_{i}^{t}/{u_{i}^{t}}\geq 1$  for all $t\geq 0$ and $i\in\{2,4\}$. 
Define the filtration
$\mathscr{F}_{t}:=\sigma\left(u^{0},w^{0},\gamma_{1}^{1},\ldots\gamma_{4}^{1},\ldots,
\gamma_{1}^{t},\ldots,\gamma_{4}^{t}\right)$. Then the following coupling construction will be 
used to define the non-increasing $\mathscr{F}_{t}$-measurable process $R_{t}$, with the 
property $R_{t}\geq\underset{i}{\max}\left\{ \frac{w_{i}^{t}}{u_{i}^{t}}\right\}$. 
Note that $u^{t}=w^{t}$ if $R_{t}=1$.

Given $u^{0}\preceq w^{0}$, we shall define two auxiliary processes $\tilde{v}$ and $v$.
Let $v^0:=w^0$, so that $\frac{u_{2}^{0}}{u_{4}^{0}}=\frac{v_{2}^{0}}{v_{4}^{0}}$. 
Let $R_0:=\frac{v_2^0}{u_2^0}$ ($=\frac{v^0_4}{u^0_4}$).  
For each $t\geq 0$, we already have (recall (\ref{eq:mc2}))  
\[
u^{t+1} \;=\; \left(u_{2}^{t+1},u_{4}^{t+1}\right)\; :=\;  F^{t+1}(u^t_2,u^t_4)   \;=\;
\left(\frac{\gamma_{2}^{t+1}}{\frac{\gamma_{1}^{t+1}}{x+u_{2}^{t}}+\frac{\gamma_{3}^{t+1}}{u_{2}^{t}+u_{4}^{t}}},\, \frac{\gamma_{4}^{t+1}}{\frac{\gamma_{3}^{t+1}}{u_{2}^{t}+u_{4}^{t}}+b}\right) \,.
\]
For each $t\geq 0$, we recursively define
\[
\tilde{v}^{t+1} \;=\; \left(\tilde{v}_{2}^{t+1},\tilde{v}_{4}^{t+1}\right)  \; := \;  F^{t+1}(v^t_2,v^t_4)  \;=\;
\left(\frac{\gamma_{2}^{t+1}}{\frac{\gamma_{1}^{t+1}}{x+v_{2}^{t}}+\frac{\gamma_{3}^{t+1}}{v_{2}^{t}+v_{4}^{t}}},\, \frac{\gamma_{4}^{t+1}}{\frac{\gamma_{3}^{t+1}}{v_{2}^{t}+v_{4}^{t}}+b}\right) \,,
\]
\[
    R_{t+1} :=  \max\left\{ \frac{\tilde{v}_{2}^{t+1}}{u_{2}^{t+1}},\frac{\tilde{v}_{4}^{t+1}}{u_{4}^{t+1}}
       \right\}  \,,   \hspace{10mm}\hbox{and} 
\]
\begin{equation}
   v^{t+1}  \;=\;
\left(v_{2}^{t+1},v_{4}^{t+1}\right):= \left(R_{t+1}u_{2}^{t+1},R_{t+1} u_{4}^{t+1}\right)\,.\label{eq:ratiomethod}
\end{equation}
Note that unlike $u^{t}$, the process $v^{t}$ is not a Markov chain. Observe also that
equality of ratios is preserved: $\frac{u_{2}^{t+1}}{u_{4}^{t+1}}=\frac{v_{2}^{t+1}}{v_{4}^{t+1}}$, 
and $\frac{v_{2}^{t+1}}{u_{2}^{t+1}}=\frac{v_{4}^{t+1}}{u_{4}^{t+1}} = R_{t+1}$.

Recall that $w^0=v^0$ and $w^{t+1}=F^{t+1}(w_t)$ for $t\geq 0$.  Then by induction,
the monotonicity of 
$F$ guarantees that $u^t\preceq w^t \preceq \tilde{v}^t \preceq v^t$ for every $t$.  That is, 
the process $v^t$ dominates a copy of the Markov chain started at $w^{0}$
and coupled uniformly with $u^{t}$.

Before deriving properties of $R_t$, we state
the following elementary calculus lemma.

\begin{lem}
\label{lem:lemma4}
Suppose that $0<a<b$. Then $g(\mathtt{x},\mathsf{\mathtt{y}}):=
(\frac{\mathtt{x}}{b}+\mathtt{y}$)/$(\frac{\mathtt{x}}{a}+\mathtt{y}$)
is decreasing in $\mathtt{x}$ and increasing in $\mathtt{y}$, for
all $\mathtt{x},\mathtt{y}>0$.
\end{lem}

\noindent We can now show that $\{R_t\}$ is \textit{non-increasing}. Let 
\[
Q_{t} \;\; := \;\;  \max \left\{  \frac{\gamma_{3}^{t+1}+bu_{4}^{t}}{\gamma_{3}^{t+1}+bv_{4}^{t}}
\,, \,
\frac{\gamma_{3}^{t+1}+\frac{\gamma_{1}^{t+1}}{1+\frac{x}{u_{2}^{t}}}}{\gamma_{3}^{t+1}+\frac{\gamma_{1}^{t+1}}{1+\frac{x}{v_{2}^{t}}}}  \right\}.
\]
Then
\begin{lem}
\label{lem:lemmaRnonincreasing}
$R_{t+1}\leq Q_{t}R_{t}$ and $Q_{t} \leq 1$.
\end{lem}
\begin{proof}
Since $u^t\preceq v^t$, it is immediate that $Q_{t} \leq 1$.
And by Lemma \ref{lem:lemma4}, we have 
\begin{equation}
\begin{aligned}
  R_{t+1}  & =\; \max\left\{ \frac{\frac{\gamma_{3}^{t+1}}{u_{2}^{t}+u_{4}^{t}}+b}{\frac{\gamma_{3}^{t+1}}{v_{2}^{t}+v_{4}^{t}}+b}\; ,\, \frac{\frac{\gamma_{3}^{t+1}}{u_{2}^{t}+u_{4}^{t}}+\frac{\gamma_{1}^{t+1}}{u_{2}^{t}+x}}{\frac{\gamma_{3}^{t+1}}{v_{2}^{t}+v_{4}^{t}}+\frac{\gamma_{1}^{t+1}}{v_{2}^{t}+x}} \right\} \\
 & = \;\frac{v_{2}^{t}}{u_{2}^{t}}\cdot \max\left\{ \left(\frac{\frac{\gamma_{3}^{t+1}}{\frac{u_{2}^{t}}{u_{4}^{t}}+1}+bu_{4}^{t}}{\frac{\gamma_{3}^{t+1}}{\frac{u_{2}^{t}}{u_{4}^{t}}+1}+bv_{4}^{t}}\right),\left(\frac{\frac{\gamma_{3}^{t+1}}{\frac{u_{4}^{t}}{u_{2}^{t}}+1}+\frac{\gamma_{1}^{t+1}}{1+\frac{x}{u_{2}^{t}}}}{\frac{\gamma_{3}^{t+1}}{\frac{u_{4}^{t}}{u_{2}^{t}}+1}+\frac{\gamma_{1}^{t+1}}{1+\frac{x}{v_{2}^{t}}}}\right)\right\} \\
 & \leq \; R_t\,Q_{t}    \,.
\end{aligned}
\label{eq:ratiocalc}
\end{equation}
\end{proof}
\noindent Lemma \ref{lem:lemmaRnonincreasing} shows that the sequence $\{R_t\}$ 
is non-increasing when $u^0\leq v^0$, and 
\[   
   \mathbb{E}[R_{t+1}] \; \leq \;  R_{0}\mathbb{E}\left[\prod_{j=0}^{t}Q_{j}\right]  \,.
\]

\noindent The next lemma shows that $\mathbb{P}\left[u^{[t+1]C}\neq w^{[t+1]C}
\left|\mathscr{F}_{t}\right.\right]$ is small if $R_t$ is close to 1. 

\begin{lem}
\label{lem:lemma3}
Assume $u^0\preceq w^0$.  For one-shot coupling
at time $t+1$, we have 
\[
\mathbb{P}\left[ \left. u^{[t+1]C}\neq w^{[t+1]C}\right|\mathscr{F}_{t} \right] \; \leq \;
 1-R_{t}^{-(a_{2}+a_{3} +a_4+a_5)} .
\]
\end{lem}

\begin{proof}
For $i\in\left\{2,4\right\}$ and $\mathscr{G}_{t}:=\sigma\left(\mathscr{F}_{t},\gamma_{1}^{t+1},\gamma_{3}^{t+1}\right)$, let $f_{u_{i}}(y)$ and $f_{w_{i}}(y)$
be the conditional density functions of $u_{i}^{t+1}$ and $w_{i}^{t+1}$
given $\mathscr{G}_{t}$, as in our description of one-shot coupling. By (\ref{eq:mc2}), these 
are gamma densities with shape parameters $a_{i}+a_{i+1}$, and inverse scale parameters 
$\Delta_{u,2}^{t+1}:=\frac{\gamma_{1}^{t+1}}{x+u_{2}^{t}}+\frac{\gamma_{3}^{t+1}}{u_{2}^{t}+u_{4}^{t}}$ and $\Delta_{u,4}^{t+1}:=b+\frac{\gamma_{3}^{t+1}}{u_{2}^{t}+u_{4}^{t}}$, with 
$\Delta_{w,2}^{t+1}$ and $\Delta_{w,4}^{t+1}$ defined similarly. 
Observe that $\Delta_{u,i}^{t+1}\geq \Delta_{w,i}^{t+1}$.  Then for all $y>0$,
\[
f_{w_{i}}(y) \; \geq \; \left(\frac{\Delta_{w,i}^{t+1}}{\Delta_{u,i}^{t+1}}\right)^{a_{i}+a_{i+1}}
  f_{u_{i}}(y)
\]
and therefore
\[
\min\left\{ f_{u_{i}}(y),f_{w_{i}}(y)\right\} \; \geq\; \left(\frac{\Delta_{w,i}^{t+1}}{\Delta_{u,i}^{t+1}}\right)^{a_{i}+a_{i+1}} f_{u_{i}}(y) \,.
\]
For the uniform coupling we have $u_i^{t+1}=\gamma_i^{t+1}/\Delta_{u,i}^{t+1}$ and
$w_i^{t+1}=\gamma_i^{t+1}/\Delta_{w,i}^{t+1}$, and hence
\[    \frac{\Delta_{u,i}^{t+1}}{\Delta_{w,i}^{t+1}}   \;=\;
    \frac{w_i^{t+1}}{u_i^{t+1}}    \;\leq \;  \frac{v_i^{t+1}}{u_i^{t+1}}   \;=\; R_{t+1}  \;\leq R_t \,.
\]
By our construction of the one-shot coupling, 
\begin{eqnarray*}
\mathbb{P}\left[ \left. u_{i}^{[t+1]C}\neq w_{i}^{[t+1]C}\right|\mathscr{G}_{t}\right] & = &
   1-\int \min\left\{ f_{u_{i}}(y),f_{w_{i}}(y)\right\} dy\\
 & \leq & 1-\left(\frac{\Delta_{w,i}^{t+1}}{\Delta_{u,i}^{t+1}}\right)^{a_{i}+a_{i+1}}\\
 & \leq & 1-R_{t}^{-a_{i}-a_{i+1}}.
\end{eqnarray*}
Since the final bound
is independent of $(\gamma_{1}^{t+1},\gamma_{3}^{t+1}$),
we also get $\mathbb{P}\left[\left. u_{i}^{[t+1]C}\neq w_{i}^{[t+1]C}\right|\mathscr{F}_{t}\right]
\,\leq\,1-R_{t}^{-a_{i}-a_{i+1}}$.
Therefore
\begin{eqnarray*}
\mathbb{P}\left[\left. u^{[t+1]C}\neq w^{[t+1]C}\right|\mathscr{F}_{t}\right] 
& = & \mathbb{P}\left[ \left. \cup_{i}\left\{ u_{i}^{[t+1]C}\neq w_{i}^{[t+1]C}\right\}
   \right|\mathscr{F}_{t}\right]\\
 & = & 1-\prod_{i=2,4}\mathbb{P}\left[ \left. \left\{ u_{i}^{[t+1]C}=w_{i}^{[t+1]C}\right\}
   \right|\mathscr{F}_{t}\right]\\
 & \leq & 1-R_{t}^{-a_{2}-a_{3}}R_t^{-a_4-a_5}.  
\end{eqnarray*}
\end{proof}

As we have seen, our ratio $R_{t}$ satisfies $R_{t}\geq \max\left\{ \frac{w_{i}^{t}}{u_{i}^{t}}\right\}$, which is the condition stated at the beginning of Section \ref{sec:RatioRt}.
Our aim now is to show that $R_{t}$ converges to $1$ at a geometric rate,    
or more explicitly to obtain an expression of the form 
\[
\mathbb{E}[R_{t+1}]\;\leq \; 1+C_{R_{0}}\prod_{j=1}^{t+1}r_{j}
\]
where $r_{j}<1$ and $r_{j}$ is ``frequently'' bounded from above by
some $r<1$ (the exact meaning of this will become apparent following
the definition of $\bar{S}_{t}$ in (\ref{eq:Sbar})). Note that in
order to achieve this, it suffices to have for all $t\geq0$ 
\begin{equation}
\mathbb{E}[Q_{t}R_{t}]\leq r_{t+1}\left(\mathbb{E}[R_{t}]-1\right)+1\label{eq:contraction1}
\end{equation}
Recall that $\mathscr{F}_{t}:=\sigma\left(u^{0},v^{0},\gamma_{1}^{1},\ldots,\gamma_{4}^{1},\ldots,\gamma_{1}^{t},\ldots,\gamma_{4}^{t}\right)$.
We can consider (\ref{eq:contraction1}) by conditioning on this filtration
\begin{equation}
\begin{aligned}\mathbb{E}[Q_{t}R_{t}] & =\mathbb{E}\left[R_{t}\mathbb{E}\left[Q_{t}\mid\mathscr{F}_{t}\right]\right]\end{aligned}
\label{eq:contraction2}
\end{equation}
and we may approximate $\mathbb{E}\left[Q_{t}\mid\mathscr{F}_{t}\right]$
with the aid of the following lemma.

\begin{lem}
\label{lem:lemma5}
Let $\mu_{1}=\mathbb{E}\left[\gamma_{3}\right]=a_{3}+a_{4}$ and $\mu_{2}=\mathbb{E}\left[\gamma_{1}-\frac{1}{3}\right]=a_{1}+a_{2}-\frac{1}{3}$, and let 
$\rrr{r}_{t}=1-1/\max\left\{ \left(\frac{4\mu_{1}}{\mu_{2}}+4\right)\left(\frac{u_{2}^{t}}{x}+\frac{x}{v_{2}^{t}}+2\right),4+\frac{4\mu_{1}}{bv_{4}^{t}}\right\} $.  
Let $S$ be a $\mathscr{F}_{t}$-measurable stopping time. Then 
\[
  \mathbb{E}\left[Q_{S}R_{S}\right] \; \leq \; \mathbb{E}\left[\rrr{r}_{S}\left(R_{S}-1\right)\right]+1 \,.
\] 
\end{lem}
\begin{proof}
By \cite{key-4} we have $\mathbb{P}\left[\gamma_{3}\leq\mu_{1}\right]\geq\frac{1}{2}$
and $\mathbb{P}\left[\gamma_{1}\geq\mu_{2}\right]\geq\frac{1}{2}$.  
Hence by Lemma \ref{lem:lemma4}, for any $t$, the
probability is at least $\frac{1}{4}$ that $Q_{t}\leq \max\left\{ \left(\frac{\mu_{1}+\frac{\mu_{2}}{1+\frac{x}{u_{2}^{t}}}}{\mu_{1}+\frac{\mu_{2}}{1+\frac{x}{v_{2}^{t}}}}\right),\left(\frac{\mu_{1}+bu_{4}^{t}}{\mu_{1}+bv_{4}^{t}}\right)\right\} $.  Then 
\begin{eqnarray}
\mathbb{E}\left[Q_{S}\mid\mathscr{F}_{S}\right] & \leq & \frac{1}{4}\cdot 
\max\left\{ \left(\frac{\mu_{1}+\frac{\mu_{2}}{1+\frac{x}{u_{2}^{S}}}}{\mu_{1}+\frac{\mu_{2}}{1+\frac{x}{v_{2}^{S}}}}\right),\left(\frac{\mu_{1}+bu_{4}^{S}}{\mu_{1}+bv_{4}^{S}}\right)\right\} +1\cdot\frac{3}{4}\nonumber \\
 & = & \frac{1}{4}\cdot \max\left\{ \left(1-\frac{\frac{1}{1+\frac{x}{v_{2}^{S}}}-\frac{1}{1+\frac{x}{u_{2}^{S}}}}{\frac{\mu_{1}}{\mu_{2}}+\frac{1}{1+\frac{x}{v_{2}^{S}}}}\right),\left(1-\frac{bv_{4}^{S}-bu_{4}^{S}}{\mu_{1}+bv_{4}^{S}}\right)\right\} +\frac{3}{4}\nonumber \\
 & \leq & \frac{1}{4}\cdot \max\left\{ 1-\frac{\left(1-\frac{1}{R_{S}}\right)}{\left(\frac{\mu_{1}}{\mu_{2}}+1\right)\left(1+\frac{x}{v_{2}^{S}}\right)\left(1+\frac{x}{u_{2}^{S}}\right)\frac{u_{2}^{S}}{x}},1-\frac{\left(1-\frac{1}{R_{S}}\right)}{1+\frac{\mu_{1}}{bv_{4}^{S}}}\right\} +\frac{3}{4}\nonumber \\
 & \leq & 1-\frac{\left(1-\frac{1}{R_{S}}\right)}{\max\left\{ \left(\frac{4\mu_{1}}{\mu_{2}}+4\right)\left(\frac{u_{2}^{S}}{x}+\frac{x}{v_{2}^{S}}+2\right),4+\frac{4\mu_{1}}{bv_{4}^{S}}\right\} }\nonumber \\
 & = & \rrr{r}_{S}+\frac{1-\rrr{r}_{S}}{R_{S}}\label{eq:rexpression}
\end{eqnarray}
Substituting this into (\ref{eq:contraction2}), we get the desired
result.
\end{proof}

Our task in the next section  will be to show that we frequently have $\rrr{r}_{t}\leq r$
for some $r<1$, which by Lemma \ref{lem:lemma5}
will result in an expression of the form given by (\ref{eq:contraction1}).

\section{\label{sub:Super-martingale}Auxiliary processes with drift conditions}

We begin by stating the first of three assumptions, all of which will be justified in the next section. The assumptions are on the existence of certain auxiliary processes that will be used to bound to the random
part of $\rrr{r}_{t}$, namely $\max\left\{ \left(\frac{4\mu_{1}}{\mu_{2}}+4\right)\left(\frac{u_{2}^{t}}{x}+\frac{x}{v_{2}^{t}}+2\right),4+\frac{4\mu_{1}}{bv_{4}^{t}}\right\} $. We will show that frequently (a positive proportion of time) these processes are bounded by a constant, which by Lemma \ref{lem:lemma5} implies that $\rrr{r_{t}}$ is frequently bounded by some $r<1$.  

The first two assumptions are conditions on general random processes.

\begin{assmpt}
\label{assmpt:A1}
Let $\vec{X}_{t} $ be a Markov chain taking values in $\mathbb{R}^d$, adapted to $\mathscr{F}_{t}$. Let $J_{t}=J\left(\vec{X}_{t}\right)$, where $J$ is a non-negative, deterministic function. Assume that there exist constants $C>0$ and $A \in \left[0,1\right)$ such that $\mathbb{E}\left[J_{t+1} \left\vert \mathscr{F}_{t} \right. \right] \leq AJ_{t}+ C$ for all $t\geq 0$.
\end{assmpt} 

\noindent Let $\eta = 2C/\left(1-A\right)$ and $\beta = \left(1+A\right)/2$, and observe that if $J_{t}\geq \eta$, $\mathbb{E}\left[J_{t+1} \left\vert \mathscr{F}_{t} \right. \right] \leq \beta J_{t}$. 

\begin{lem}
\label{lem:lemma8}
Suppose that Assumption \ref{assmpt:A1} holds, and let 
\begin{eqnarray}
   \bar{S}_{t} & := & \left\{ 1\leq i\leq t\,\vert\, J_{i}\leq\eta\right\} \,.
\label{eq:Sbar}
\end{eqnarray}  
Then for any $t$ and $k$, 
$
\mathbb{P}\left[\left|\bar{S}_{t}\right|=k\left|J_{0}\leq\eta\right.\right] \; \leq \; \binom{t}{k}\beta^{t-k}\,.
$
\end{lem}
\noindent The proof is left to Section \ref{subs:remproofs}.  The above result will be used to show
that $J_t$ 
is frequently bounded by $\eta$. Now writing 
$\ddd{u}^{t}= \left(u_{1}^{t}, u_{2}^{t}, u_{3}^{t}, u_{4}^{t} \right)$ (recall the definition of $u_{1}^{t}$ 
and $u_{3}^{t}$ from  (\ref{eq:mc4D})) and similarly for  $\ddd{v}^{t}$, we state our next assumption. 

\begin{assmpt}
\label{assmpt:A2}
Fix $N\geq 1$. Assume that for $i=1,\ldots ,N$ there exist functions $K_{i}$ such that the processes $K_{i,t}=K_{i}\left(\vec{u_{t}},\vec{v_{t}}\right)$ satisfy 
\begin{equation}
\mathbb{E}\left[K_{i,t+1}\vert\mathscr{F}_{t}\right]\; \leq\; \zeta_{i}K_{i,t}+C_{i}\label{eq:kit}
\end{equation}
for $t\geq0$,  where $\zeta_{i}<1$ and $C_{i}$ are constants.
\end{assmpt} 
\noindent Let $J_{t} = \sum_{i} K_{i,t}$. Observe that under Assumption \ref{assmpt:A2}, the process $J_{t}$ satisfies Assumption \ref{assmpt:A1}, with $A=\max\left\{ \zeta_{i}\right\}$ and $C=\sum_{i}C_{i}$. 
\begin{assmpt}
\label{assmpt:A3}
There is a process $D_{t}$ adapted to $\mathscr{F}_{t}$ such that for all $t\geq 1$
\begin{equation}
  D_{t} \;\geq \; \max\left\{ \left(\frac{4\mu_{1}}{\mu_{2}}+4\right)\left(\frac{u_{2}^{t}}{x}+
  \frac{x}{v_{2}^{t}}+2\right),4+\frac{4\mu_{1}}{bv_{4}^{t}}\right\} 
 \label{eq:dtcondit1}
\end{equation}
and
\begin{equation}
D_{t+1} \; \leq \; \omega_{N+1,t+1}+\sum_{i=1}^{N}\omega_{i,t+1}K_{i,t}
\label{eq:dtinequality}
\end{equation}
where $\left(\omega_{1,t+1},\ldots,\omega_{N+1,t+1}\right)$ is a
non-negative random vector that is i.i.d. over time $t\geq1$, measurable
w.r.t. $\mathscr{F}_{t+1}$ and independent of $\mathscr{F}_{t}$.
\end{assmpt} 
 
\noindent The reasons for the condition (\ref{eq:dtinequality}) will become apparent when we construct $D_t$. Note that $\rrr{r}_{t} \leq 1- 1/D_{t}$, which is used in the next lemma.  The idea
is that if $J_t$ is bounded, then $D_t$ is probably not too large, and $\rrr{r}_t$ is not too close to 1.
 
\begin{lem}
\label{lem:newlemma1}
Suppose that Assumptions \ref{assmpt:A2} and \ref{assmpt:A3} hold. Let $S\geq1$ be an a.s. finite stopping time adapted to $\mathscr{F}_{t}$ such that $J_{S} \leq \eta$. Let $r:=1-1/\left(\left(\theta_{1}+\ldots+\theta_{N}\right)\eta+\theta_{N+1}\right)$
and $\theta_{i}:=\mathbb{E}\left[\omega_{i,t+1}\right]$.  Then:
\begin{enumerate}
\item $\mathbb{E}\left[R_{S+2} - 1\right] \leq r\mathbb{E}\left[R_{S} - 1\right]$. 
\item More generally, if $0\leq Y \in \mathscr{F}_{S}$, then $\mathbb{E}\left[R_{S+2} - 1\right] 
\leq r\mathbb{E}\left[Y\left(R_{S} - 1\right)\right]$.
\end{enumerate}
\end{lem}

\begin{proof}
We start by observing that  $D_{S+1}\leq\eta\sum\omega_{i,S+1}+\omega_{N+1,S+1}$.
Therefore, applying Lemma \ref{lem:lemma5} we get 
\begin{eqnarray}
\mathbb{E}\left[R_{S+2}\right] & \leq & \mathbb{E}\left[Q_{S+1}R_{S+1}\right]\nonumber \\
 & \leq & \mathbb{E}\left[\rrr{r}_{S+1}\left(R_{S+1}-1\right)\right]+1
  \label{eq.plus1} \\
 & \leq & \mathbb{E}\left[\left(1-\frac{1}{D_{S+1}}\right)\left(R_{S}-1\right)\right]+1
   \hspace{10mm}\hbox{(using $R_{t+1}\leq R_t$)}
 \nonumber \\
 & \leq & \mathbb{E}\left[\left(1-\frac{1}{\eta\sum\omega_{i,S+1}+\omega_{N+1,S+1}}\right)\left(R_{S}-1\right)\right]+1\nonumber \\
 & = & \mathbb{E}\left[\left(1-\frac{1}{\eta\sum\omega_{i,S+1}+\omega_{N+1,S+1}}\right)\right]\mathbb{E}\left[\left(R_{S}-1\right)\right]+1\nonumber \\
 & \leq & r\,\mathbb{E}\left[\left(R_{S}-1\right)\right]+1
   \hspace{15mm}\hbox{(by Jensen's inequality)}
 \label{eq:rs2}
\end{eqnarray}

\noindent By a derivation identical to (\ref{eq:rs2}) we get
\begin{equation}
\mathbb{E}\left[YR_{S+2}\right]\leq r\mathbb{E}\left[Y\left(R_{S}-1\right)\right]+\mathbb{E}\left[Y\right]  \,.   
\label{eq:Y}
\end{equation}
The term $\mathbb{E}\left[Y\right]$ in the right-hand side of (\ref{eq:Y})
arises in (\ref{eq.plus1}), as a result of applying (\ref{eq:rexpression}).
\end{proof}

\noindent The following are the main results of this section. The proofs are given in Section \ref{subs:remproofs}.
\begin{lem}
\label{lem:lemma9} Suppose that Assumptions \ref{assmpt:A2} and \ref{assmpt:A3} hold.
Then in the event $\left\{ J_{0}\leq\eta\right\} $
\[
  \mathbb{E}\left[\left. R_{t+2}\mathbf{1}_{\left|\bar{S}_{t}\right|>k}\right|\mathscr{F}_{0}\right]
 -\mathbb{P}\left[\left. \left|\bar{S}_{t}\right|>k\right|\mathscr{F}_{0}\right]  \; \leq \;
   r^{\left\lceil \left(k+1\right)/2\right\rceil }\left(R_{0}-1\right)
\]
\end{lem}

\begin{cor}
\label{cor:corollary10}
Let $d\,=\,\max\left\{3,\ln(\beta|\ln \beta|\sqrt{r}/2)/\ln \beta\right\}$.  
Then in $\left\{ J_{0}\leq\eta\right\}$, we have
$\mathbb{E}\left[R_{t+2}\left|\mathscr{F}_{0}\right.\right]\,\leq\,1+3r^{t/2d}\left(R_{0}-1\right)$
for all $t>0$.
\end{cor}

Now let $T=T(s):=\min\left\{ \tau>s\left|J_{\tau}\leq\eta\right.\right\} $,
and for $s\geq0$ and $t\geq1$ define 
\[
\hat{J}_{s,t}:=\begin{cases}
J_{s+t} & s+t<T,\, J_{s}\leq\eta\\
0 & \hbox{otherwise},
\end{cases}
\]
or in other words $\hat{J}_{s,t}=\mathbf{1}_{\{J_{s}\leq\eta\}\cap\{T>s+t\}}J_{s+t}$.
The  next lemma is proved in Section \ref{subs:remproofs}.

\begin{lem}
\label{lem:lemma6}For the notation and assumptions of the preceding
paragraph, $\mathbb{E}\left[\hat{J}_{s,t+1}\left|\mathscr{F}_{s}\right.\right]\leq\beta^{t+1}\eta$
for $t\geq0$ and $s\geq0$.
\end{lem}

\subsection{Remaining proofs}
\label{subs:remproofs}

We begin by stating an easy lemma, whose proof we omit.

\begin{lem}
\label{lem:lemma7}Let $Y$ be an $\mathbb{R}^d$-valued random vector. If $A$
is an event and $B\subseteq\mathbb{R}^d$ with $\mathbb{P}[Y\in B]\neq 0$, then 
\[
    \mathbb{P}\left[A\left|Y\in B\right.\right] \;\leq\; 
   \underset{y_{0}\in B}{\sup}\mathbb{P}\left[A\left|Y=y_{0}\right.\right] \,.
\]
\end{lem}


 \begin{proof}[Proof of Lemma \ref{lem:lemma8}]
It suffices to prove that for any subset 
$\left\{ c_{1},c_{2},\ldots,c_{k}\right\} \subseteq\left\{ 1,\ldots,t\right\}$, 
\[
\mathbb{P}\left[\bar{S}_{t}=\left\{ c_{1},c_{2},\ldots,c_{k}\right\} \,\vert\, J_{0}\leq\eta\right]
    \;\leq\;\beta^{t-k}.
\]
Fix such a subset.
Let $A=\left\{ (\ddd{u},\ddd{v})\in\mathbb{R}_{+}^{4}\times\mathbb{R}_+^4\, : \, 
J(\ddd{u},\ddd{v})\leq\eta\right\}$,
and let $I\subseteq\left\{ 0,1,\ldots,k\right\} $ be those indices $i$
that satisfy $c_{i+1}>c_{i}+1$, where by convention we set $c_{0}=0$
and $c_{k+1}=t{+}1$.   
For $i\in I$, let  $B_{i}=\left\{ J_{c_{i}+1}>\eta,\ldots,J_{c_{i+1}-1}>\eta\right\} $.
By Lemma \ref{lem:lemma7}, 
\begin{eqnarray*}
\mathbb{P}\left[B_{i}\,\left|\,J_{c_{i}}\leq\eta\right.\right] & \leq & \underset{y\in A}{\sup}\,\mathbb{P}\left[B_{i}\left|(\ddd{u}^{c_{i}},\ddd{v}^{c_{i}})=y\right.\right] \,.
\end{eqnarray*}
Since $J_{c_{i}}$ is determined by the values $\left(\ddd{u}^{c_{i}},\ddd{v}^{c_{i}}\right)$,
it follows by the same reasoning and the Markov property that also
for any event $G_{c_{i}-1}\in\mathscr{F}_{c_{i}-1}$ 
\begin{equation}
   \mathbb{P}\left[B_{i}\left|J_{c_{i}}\leq\eta,\,G_{c_{i}-1}\right.\right]\;\leq\;
    \underset{y\in A}{\sup}\,\mathbb{P}\left[B_{i}\left|(\ddd{u}^{c_{i}},\ddd{v}^{c_{i}})=y\right.\right] \,.
 \label{eq:Bi}
\end{equation}
Observe also that if $I=\left\{ i[1],\ldots i[m]\right\} $ for some $m\leq k+1$, then 
\[
   \sum_{j=1}^{m}\left(c_{i[j]+1}-c_{i[j]}-1\right) \;\; = \;\; \left|\left\{ 1,\ldots,t\right\} \backslash
   \left\{ c_{1},c_{2},\ldots,c_{k}\right\} \right|  \;\; = \;\; t-k \,.
\]
Hence we get 
\begin{eqnarray}
\lefteqn{\mathbb{P}\left[\bar{S}_{t}=\left\{ c_{1},c_{2},\ldots,c_{k}\right\} \left|J_{0}\leq\eta\right.\right]}
  \nonumber \\ 
  & = & \mathbb{P}\left[\left\{ J_{c_{1}}\leq\eta,\ldots,J_{c_{k}}\leq\eta\right\} \cap\left\{ \cap_{i\in I}B_{i}\right\} \left|J_{0}\leq\eta\right.\right]
    \label{eq.arrAA}    \\
 & \leq & \mathbb{P}\left[\left\{ J_{c_{i[1]}}\leq\eta,\ldots,J_{c_{i[m]}}\leq\eta\right\} \cap\left\{ \cap_{i\in I}B_{i}\right\} \left|J_{0}\leq\eta\right.\right]
   \label{eq.arrAB}  \\
 & = & \mathbb{P}\left[B_{i_{m}}\left|\left\{ J_{c_{i[1]}}\leq\eta\ldots,J_{c_{i[m]}}\leq\eta\right\} \cap\left\{ \cap_{j=1}^{m-1}B_{i[j]}\right\} ,J_{0}\leq\eta\right.\right]\times
    \nonumber \\
 &  & \hspace{4mm} \mathbb{P}\left[\left\{ J_{c_{i[1]}}\leq\eta,\ldots,J_{c_{i[m]}}\leq\eta\right\} \cap\left\{ \cap_{j=1}^{m-1}B_{i[j]}\right\} \left|J_{0}\leq\eta\right.\right]
   \nonumber \\
 & \leq & \underset{y\in A}{\sup}\,\mathbb{P}\left[B_{i[m]}\left|(\ddd{u}^{c_{i[m]}},\ddd{v}^{c_{i[m]}})
   =y\right.\right]  \times  
     \nonumber \\
  &   & \hspace{4mm}
 \mathbb{P}\left[\left\{ J_{c_{i[1]}}\leq\eta,\ldots,J_{c_{i[m-1]}}\leq\eta\right\} \cap\left\{ \cap_{j=1}^{m-1}B_{i[j]}\right\} \left|J_{0}\leq\eta\right.\right]
   \nonumber \\
 & \vdots   \nonumber \\
 & \leq & \prod_{j=1}^{m}\underset{y\in A}{\sup}\,\mathbb{P}\left[\left. B_{i[j]}\right| \,
   (\ddd{u}^{c_{i[j]}},\ddd{v}^{c_{i[j]}})=y \right]
    \nonumber \\
 & \leq & \prod_{j=1}^{m}\underset{y\in A}{\sup}\,\mathbb{P}\left[\left. \hat{J}_{c_{i[j]},c_{(i[j]+1)}-c_{i[j]}-1}>\eta \,\right|\, (\ddd{u}^{c_{i[j]}},\ddd{v}^{c_{i[j]}})=y\right]
  \nonumber   \\
 & \leq & \beta^{c_{i[1]+1}-c_{i[1]}-1}\cdots\beta^{c_{i[m]+1}-c_{i[m]}-1}
    \hspace{10mm}\hbox{(by Lemma \ref{lem:lemma6} and Markov's inequality)}
   \nonumber  \\
 & = & \beta^{t-k} \,.     \nonumber 
\end{eqnarray}
We remark that when $i[1]=0$, the event $J_{c_{i[1]}}\leq\eta$
appears in (\ref{eq.arrAB}) but not in (\ref{eq.arrAA}). This is justified
because in this case $J_{c_{i[1]}}=J_{0}$, and we are conditioning on $J_{0}\leq\eta$. 
\end{proof}
 
\begin{proof}[Proof of Lemma \ref{lem:lemma9}]
Let $\tau_{0}=0$ and $\left\{ \tau_{i}\right\} \subseteq\left\{ 1,2,\ldots\right\} $
be those times for which $J_{\tau_{i}}\leq\eta$. Then by $(\ref{eq:Y})$
with $Y=\mathbf{1}_{\tau_{k+1}\leq t}$ and $S=\tau_{k+1}$
\begin{eqnarray*}
\mathbb{E}\left[R_{t+2}\mathbf{1}_{\left|\bar{S}_{t}\right|>k}\left|\mathscr{F}_{0}\right.\right] & = & \mathbb{E}\left[R_{t+2}\mathbf{1}_{\tau_{k+1}\leq t}\left|\mathscr{F}_{0}\right.\right]\\
 & \leq & \mathbb{E}\left[R_{\tau_{k+1}+2}\mathbf{1}_{\tau_{k+1}\leq t}\left|\mathscr{F}_{0}\right.\right]\\
 & \leq & r\mathbb{E}\left[\mathbf{1}_{\tau_{k+1}\leq t}\left(R_{\tau_{k+1}}-1\right)\left|\mathscr{F}_{0}\right.\right]+\mathbb{P}\left[\left|\bar{S}_{t}\right|>k\left|\mathscr{F}_{0}\right.\right]\\
 & \leq & r\mathbb{E}\left[\mathbf{1}_{\tau_{k-1}\leq t}\left(R_{\tau_{k-1}+2}-1\right)\left|\mathscr{F}_{0}\right.\right]+\mathbb{P}\left[\left|\bar{S}_{t}\right|>k\left|\mathscr{F}_{0}\right.\right]
\end{eqnarray*}
The last inequality uses the fact that $\mathbf{1}_{\tau_{k+1}\leq t}\leq\mathbf{1}_{\tau_{k-1}\leq t}$
and $R_{\tau_{k+1}}\leq R_{\tau_{k-1}+2}$. This then leads to the
first step in an inductive argument: 
\begin{eqnarray}
   \mathbb{E}\left[R_{\tau_{k+1}+2}\mathbf{1}_{\tau_{k+1}\leq t}\left|\mathscr{F}_{0}\right.\right]
   & - & \mathbb{P}\left[\left|\bar{S}_{t}\right|>k\left|\mathscr{F}_{0}\right.\right] 
    \label{eq:inductRt}  \\  
   & \leq & r\left(\mathbb{E}\left[R_{\tau_{k-1}+2}\mathbf{1}_{\tau_{k-1}\leq t}\left|
   \mathscr{F}_{0}\right.\right]-\mathbb{P}\left[\left|\bar{S}_{t}\right|>k-2\left|\mathscr{F}_{0}\right.\right]\right)
    \nonumber 
\end{eqnarray}
Proceeding in this manner, we claim that we get
\[
\mathbb{E}\left[R_{\tau_{k+1}+2}\mathbf{1}_{\tau_{k+1}\leq t}\left|\mathscr{F}_{0}\right.\right]-\mathbb{P}\left[\left|\bar{S}_{t}\right|>k\left|\mathscr{F}_{0}\right.\right]
  \; \leq \;  r^{\left\lceil \left(k+1\right)/2\right\rceil }\left(R_{0}-1\right) \,.
\]
The ceiling function in the exponent $\left\lceil \left(k+1\right)/2\right\rceil $
is immediate whenever $k+1$ is even. If on the other hand $k+1$
is odd, then by (\ref{eq:inductRt}) and (\ref{eq:Y}) we have
\begin{eqnarray*}
\mathbb{E}\left[R_{\tau_{k+1}+2}\mathbf{1}_{\tau_{k+1}\leq t}\left|\mathscr{F}_{0}\right.\right]-\mathbb{P}\left[\left|\bar{S}_{t}\right|>k\left|\mathscr{F}_{0}\right.\right] & \leq & r^{\left\lfloor \left(k+1\right)/2\right\rfloor }\mathbb{E}\left[\mathbf{1}_{\tau_{1}\leq t}\left(R_{\tau_{1}+2}-1\right)\left|\mathscr{F}_{0}\right.\right]\\
 & \leq & r^{\left\lfloor \left(k+1\right)/2\right\rfloor }r\mathbb{E}\left[\mathbf{1}_{\tau_{1}\leq t}\left(R_{\tau_{1}}-1\right)\left|\mathscr{F}_{0}\right.\right]\\
 & \leq & r^{\left\lfloor \left(k+1\right)/2\right\rfloor +1}\left(R_{0}-1\right) \,.
\end{eqnarray*}
\end{proof} 
 
\begin{proof}[Proof of Corollary \ref{cor:corollary10}]
For any $k< t$, we deduce from Lemmas \ref{lem:lemma8} and \ref{lem:lemma9} that
\begin{eqnarray}
\mathbb{E}\left[R_{t+2}\left|\mathscr{F}_{0}\right.\right] & = & \mathbb{E}\left[R_{t+2}\mathbf{1}_{\left|\bar{S}_{t}\right|>k}\left|\mathscr{F}_{0}\right.\right]+\mathbb{E}\left[R_{t+2}\mathbf{1}_{\left|\bar{S}_{t}\right|\leq k}\left|\mathscr{F}_{0}\right.\right]\nonumber \\
 & \leq & r^{\left\lceil \left(k+1\right)/2\right\rceil }\left(R_{0}-1\right)+\mathbb{P}\left[\left|\bar{S}_{t}\right|>k\left|\mathscr{F}_{0}\right.\right]+\mathbb{E}\left[R_{0}\mathbf{1}_{\left|\bar{S}_{t}\right|\leq k}\left|\mathscr{F}_{0}\right.\right]\nonumber \\
 & \leq & r^{\left\lceil \left(k+1\right)/2\right\rceil }\left(R_{0}-1\right)+\mathbb{P}\left[\left|\bar{S}_{t}\right|>k\left|\mathscr{F}_{0}\right.\right]\nonumber \\ & & +\left(R_{0}-1\right)\mathbb{P}\left[\left|\bar{S}_{t}\right|\leq k\left|\mathscr{F}_{0}\right.\right]+\mathbb{P}\left[\left|\bar{S}_{t}\right|\leq k\left|\mathscr{F}_{0}\right.\right]\nonumber \\
 & \leq & 1+\left(R_{0}-1\right)\left(r^{\left\lceil \left(k+1\right)/2\right\rceil }+\sum_{j=0}^{k}\binom{t}{j}\beta^{t-j}\right) \,.
\label{eq:Rbound1}
\end{eqnarray}

Henceforth, let $k=\left\lfloor \frac{t}{d}\right\rfloor$.  Since $k\leq t/3$, 
we have $\binom{t}{j}\leq\frac{1}{2}\binom{t}{j+1}$ for $j<k$ and hence
$\sum_{j=0}^{k}\binom{t}{j}\beta^{t-j}\leq2\binom{t}{k}\beta^{t-k}$.
Next, note that $\binom{t}{k}q^k(1-q)^{t-k}\leq 1$ whenever $0<q<1$. Taking $q=1/d$, we get
\begin{eqnarray}
  \nonumber
  \binom{t}{k}  \; \leq \; d^k\,\left(1-\frac{1}{d}\right)^{-(t-k)}    
     & =  &   \frac{d^t}{(d-1)^{t- k}}  \;\;\leq \; \; \frac{d^t}{(d-1)^{t- t/d}} \\
     \label{eq-binbd}
     & = & \left[d\,\left(1+\frac{1}{d-1}\right)^{d-1}\right]^{t/d}   \;<\;  (d\,e)^{t/d}  \,.
\end{eqnarray}

By calculus, we have $y\,\beta^{y}\,\leq \, 2\,\beta^{y/2}/(e|\ln \beta|)$ for all $y>0$.  
Combining this with results of the preceding paragraph, we obtain
\begin{eqnarray*}
    r^{\left\lceil (k+1)/2\right\rceil }\,+\,  \sum_{j=0}^{k}\binom{t}{j}\beta^{t-j}
       & \leq & r^{(k+1)/2}+2\binom{t}{k}\beta^{t-k}  \\
       & \leq & r^{t/2d}+2\left(ed\beta^{d-1}\right)^{t/d}   \\
    & \leq & r^{t/2d}+2\left(\frac{2}{\beta|\ln \beta|}\beta^{d}\right)^{t/d}   \\
    & \leq  & 3r^{t/2d}.
\end{eqnarray*}
Together with (\ref{eq:Rbound1}), this proves the desired bound.
\end{proof}
\begin{proof}[Proof of Lemma \ref{lem:lemma6}]
Observe that for $t\geq1$, 
\begin{eqnarray*}
\mathbb{E}\left[\hat{J}_{s,t+1}\left|\mathscr{F}_{s+t}\right.\right] & = & \mathbf{1}_{\{J_{s}\leq\eta\}\cap\{T\leq s+t\}}\mathbb{E}\left[\hat{J}_{s,t+1}\left|\mathscr{F}_{s+t}\right.\right]+\mathbf{1}_{\{J_{s}\leq\eta\}\cap\{T>s+t\}}\mathbb{E}\left[\hat{J}_{s,t+1}\left|\mathscr{F}_{s+t}\right.\right]\\
 & = & 0+\mathbf{1}_{\{J_{s}\leq\eta\}\cap\{T>s+t\}}\mathbb{E}\left[\mathbf{1}_{\{T>s+t+1\}}J_{s+t+1}\left|\mathscr{F}_{s+t}\right.\right]\\
 & \leq & \mathbf{1}_{\{J_{s}\leq\eta\}\cap\{T>s+t\}}\mathbb{E}\left[\mathbf{1}_{\{J_{s+t}>\eta\}}J_{s+t+1}\left|\mathscr{F}_{s+t}\right.\right]\\
 & \leq & \mathbf{1}_{\{J_{s}\leq\eta\}\cap\{T>s+t\}}\beta J_{s+t}\\
 & = & \beta\hat{J}_{s,t}  \,.
\end{eqnarray*}
Proceeding inductively, it follows that 
$
\mathbb{E}\left[\hat{J}_{s,t+1}\left|\mathscr{F}_{s}\right.\right] \; \leq \; \mathbb{E}\left[\beta^{t}\hat{J}_{s,1}\left|\mathscr{F}_{s}\right.\right]  \,.
$
Finally, 
\begin{eqnarray*}
\mathbb{E}\left[\hat{J}_{s,1}\left|\mathscr{F}_{s}\right.\right] & \leq & \mathbb{E}\left[\mathbf{1}_{\{J_{s}\leq\eta\}}J_{s+1}\left|\mathscr{F}_{s}\right.\right]\\
 & = & \mathbb{E}\left[J_{s+1}\left|\mathscr{F}_{s}\right.\right]}\, \mathbf{1}_{\{J_{s}\leq\eta\}\\
 & \leq & \left(\sum_{i}c_{i}+\max\left\{ \zeta_{i}\right\} J_{s}\right)\mathbf{1}_{\{J_{s}\leq\eta\}}\\
 & \leq & \left(\sum_{i}c_{i}+\max\left\{ \zeta_{i}\right\} \eta\right)\mathbf{1}_{\{J_{s}\leq\eta\}}\\
 & \leq & \beta\eta  \,.
\end{eqnarray*}
\end{proof}

\begin{rem*}
\label{rem-Jeta}
If it is uncertain whether $J_{s}\leq\eta$, we can still define 
$\check{J}_{s,t}=\mathbf{1}_{\{T>s+t\}}J_{s+t}$,
and following the proof of Lemma \ref{lem:lemma6} it is a straightforward conclusion that
\begin{equation}
   \mathbb{E}\left[\check{J}_{s,t+1}\left|\mathscr{F}_{s}\right.\right]\leq\beta^{t+1}
   \max\left\{ \eta,J_{s}\right\} \,.
\label{eq:Jslarge}
\end{equation}
\end{rem*}

 \medskip{}

\section{\label{sub:ExistenceDt}Construction of $D_{t}$}

For ease of reference, we first give the following
list of definitions for $t\geq0$ (unless otherwise indicated).

\medskip

\begin{tabular}{l}
$K_{1,t}:=u_{2}^{t}+u_{4}^{t}$ \hspace{25mm}
$K_{2,t}:=\frac{u_{3}^{t}+u_{1}^{t}+b}{\gamma_{2}^{t}+\gamma_{4}^{t}}\,,\,t\geq1$\tabularnewline
$D_{t}:=\frac{1}{x}\left(\frac{4\mu_{1}}{\mu_{2}}+4\right)\left(u_{2}^{t}+u_{4}^{t}\right)+\left(\left(\frac{4\mu_{1}}{\mu_{2}}+4\right)x+\frac{4\mu_{1}}{b}\right)\left(\frac{1}{u_{2}^{t}}+\frac{1}{u_{4}^{t}}\right)\,,\,t\geq1$ \tabularnewline
$\zeta_{1}:=\frac{a_{2}+a_{3}}{a_{1}+a_{2}+a_{3}+a_{4}-1}$ \hspace{20mm}
$\zeta_{2}:=\frac{a_{3}+a_{4}}{a_{2}+a_{3}+a_{4}+a_{5}-1}$\tabularnewline
$C_{1}:=\zeta_{1}x+\frac{a_{4}+a_{5}}{b}$ \hspace{23mm}
$C_{2}:=\frac{a_{1}+a_{2}+xb}{x\left(a_{2}+a_{3}+a_{4}+a_{5}-1\right)}$\tabularnewline
$\tilde{\omega}_{2,t+1}=2+\frac{\gamma_{2}^{t+1}}{\gamma_{4}^{t+1}}+\frac{\gamma_{4}^{t+1}}{\gamma_{2}^{t+1}}$\hspace{15mm}
$\omega_{1,t+1}:=\frac{1}{x}\left(\frac{4\mu_{1}}{\mu_{2}}+4\right)\frac{\gamma_{2}^{t+1}}{\gamma_{1}^{t+1}+\gamma_{3}^{t+1}}$\tabularnewline
$\omega_{2,t+1}:=\left(\left(\frac{4\mu_{1}}{\mu_{2}}+4\right)x+\frac{4\mu_{1}}{b}\right)\tilde{\omega}_{2,t+1}\frac{\gamma_{3}^{t+1}}{\gamma_{2}^{t+1}+\gamma_{4}^{t+1}}$ \tabularnewline
$\omega_{3,t+1}:=\frac{1}{x}\left(\frac{4\mu_{1}}{\mu_{2}}+4\right)\left(\frac{\gamma_{2}^{t+1}}{\gamma_{1}^{t+1}+\gamma_{3}^{t+1}}x+\frac{\gamma_{4}^{t+1}}{b}\right)+\left(\left(\frac{4\mu_{1}}{\mu_{2}}+4\right)x+\frac{4\mu_{1}}{b}\right)\tilde{\omega}_{2,t+1}\frac{\frac{\gamma_{1}^{t+1}}{x}+b}{\gamma_{2}^{t+1}+\gamma_{4}^{t+1}}$ \tabularnewline
\end{tabular}

\medskip

\noindent We also let $K_{2,0} = 1/\left(u_{2}^{0} + u_{4}^{0}\right)$. Note that
\begin{equation}
\max\left\{ \left(\frac{4\mu_{1}}{\mu_{2}}+4\right)\left(\frac{u_{2}^{t}}{x}+\frac{x}{v_{2}^{t}}+2\right),4+\frac{4\mu_{1}}{bv_{4}^{t}}\right\} \;\leq\;  D_{t}
\label{eq:dtineq}
\end{equation}
 where we have used the facts that $u\preceq v$ and $2\leq\frac{u}{x}+\frac{x}{u}$.
To bound the first term in the expression for $D_t$, observe that for $t\geq0$
\begin{eqnarray}
u_{2}^{t+1}+u_{4}^{t+1} & = & \frac{\gamma_{2}^{t+1}}{\frac{\gamma_{1}^{t+1}}{x+u_{2}^{t}}+\frac{\gamma_{3}^{t+1}}{u_{2}^{t}+u_{4}^{t}}}+\frac{\gamma_{4}^{t+1}}{\frac{\gamma_{3}^{t+1}}{u_{2}^{t}+u_{4}^{t}}+b}\nonumber \\
 & \leq & \frac{\gamma_{2}^{t+1}}{\gamma_{1}^{t+1}+\gamma_{3}^{t+1}}\left(u_{2}^{t}+u_{4}^{t}+x\right)+\frac{\gamma_{4}^{t+1}}{b}\label{eq:D1}
\end{eqnarray}
Therefore $\mathbb{E}\left[K_{1,t+1}\left|\mathscr{F}_{t}\right.\right]\leq\zeta_{1}K_{1,t}+C_{1}$.
Observe that since 
\[
u_{3}^{t+1} \; = \; \frac{\gamma_{3}^{t+1}}{u_{2}^{t}+u_{4}^{t}}  \;\;=\;\; \frac{\gamma_{3}^{t+1}}{\frac{\gamma_{2}^{t}}{u_{1}^{t}+u_{3}^{t}}+\frac{\gamma_{4}^{t}}{u_{3}^{t}+b}}  
  \;\leq \;
 \frac{\gamma_{3}^{t+1}}{\gamma_{2}^{t}+\gamma_{4}^{t}}\left(u_{1}^{t}+u_{3}^{t}+b\right)
 \;\;  =\;\;  \gamma_3^{t+1}K_{2,t} 
\]
for $t\geq1$, it follows that 
\begin{eqnarray}
  K_{2,t+1} & \leq & \frac{\gamma_{3}^{t+1}}{\gamma_{2}^{t+1}+\gamma_{4}^{t+1}}K_{2,t} \,+\,
  \frac{u_{1}^{t+1}+b}{\gamma_{2}^{t+1}+\gamma_{4}^{t+1}}
  \label{eq:D2}
\end{eqnarray}
and hence 
\begin{equation}
\mathbb{E}\left[K_{2,t+1}\left|\mathscr{F}_{t}\right.\right] \;\; \leq \;\; \zeta_{2}K_{2,t} \,+\,
  \mathbb{E}\left[\left. \frac{\frac{\gamma_{1}^{t+1}}{x}+b}{\gamma_{2}^{t+1}+\gamma_{4}^{t+1}}
   \right|\mathscr{F}_{t}\right] 
 \;\; \leq \;\; \zeta_{2}K_{2,t}\,+\,C_{2} \,.
 \label{eq:k2ineqn4}
\end{equation}
for $t\geq0$ (the $t{=}0$ case is immediate from the definition of $K_{2,0}$). Both $K_{1,t}$ and $K_{2,t}$ are adapted to $\mathscr{F}_{t}$ and
are in fact functions of $\ddd{u}^{t}$ for $t\geq1$ (since $\gamma_{2}^{t}+\gamma_{4}^{t}=u_{2}^{t}\left(u_{1}^{t}+u_{3}^{t}\right)+u_{4}^{t}\left(u_{3}^{t}+b\right)$).
This verifies Assumption \ref{assmpt:A2} with $N{=}2$.
Note also that 
\begin{equation}
\frac{1}{u_{2}^{t+1}}+\frac{1}{u_{4}^{t+1}} \;\; \leq \;\; \left(\frac{1}{\gamma_{2}^{t+1}}+\frac{1}{\gamma_{4}^{t+1}}\right)\left(u_{1}^{t+1}+u_{3}^{t+1}+b\right) 
 \;\; = \;\; \tilde{\omega}_{2,t+1}K_{2,t+1}
 \label{eq:D3}
\end{equation}
and $\tilde{\omega}_{2,t+1}$ is independent of $\mathscr{F}_{t}$.
By (\ref{eq:D1}), (\ref{eq:D2}) and (\ref{eq:D3}) we conclude that for $t\geq1$,
\begin{eqnarray*}
D_{t+1} & \leq & \frac{1}{x}\left(\frac{4\mu_{1}}{\mu_{2}}+4\right)\left(\frac{\gamma_{2}^{t+1}}{\gamma_{1}^{t+1}+\gamma_{3}^{t+1}}\left(K_{1,t}+x\right)+\frac{\gamma_{4}^{t+1}}{b}\right) \\ & & +\left(\left(\frac{4\mu_{1}}{\mu_{2}}+4\right)x+\frac{4\mu_{1}}{b}\right)\tilde{\omega}_{2,t+1}\left(\frac{\gamma_{3}^{t+1}}{\gamma_{2}^{t+1}+\gamma_{4}^{t+1}}K_{2,t}+\frac{\frac{\gamma_{1}^{t+1}}{x}+b}{\gamma_{2}^{t+1}+\gamma_{4}^{t+1}}\right)\\
 & \leq & \omega_{1,t+1}K_{1,t}+\omega_{2,t+1}K_{2,t}+\omega_{3,t+1}
\end{eqnarray*}
and hence $D_{t}$ satisfies Assumption \ref{assmpt:A3}.
Referring back to  Lemma \ref{lem:newlemma1},   
we obtain the rate
\begin{equation}
r \;=\; 1-\frac{1}{\left(\theta_{1}+\theta_{2}\right)\eta+\theta_{3}}
\end{equation}
where $\theta_{1},\theta_{2},\theta_{3}$ are the expected values
of $\omega_{1,t+1},\omega_{2,t+1}$ and $\omega_{3,t+1}$ respectively.

\medskip{}

We make the additional note that it is not necessary for $\left\{ K_{i,t}\right\} $
to be deterministic functions of $\left(u^{t},v^{t}\right)$. This
assumption was required to make use of the Markov property in (\ref{eq:rs2})
and (\ref{eq:Bi}), however the arguments remain true if $\left\{ K_{i,t}\right\} $
are random functions of $\left(u^{t},v^{t}\right)$ with random terms
that are independent of $\mathscr{F}_{\infty}$.
Note also that condition (\ref{eq:condition1}) guarantees that $\zeta_{1}<1$
and $\zeta_{2}<1$, as well as the finite value of all constants and
finite expectation of all random variables defined in the beginning
of this section.

\smallskip

We have now established a sufficient foundation to prove our first
theorem.

\smallskip

\begin{proof}[Proof of Theorem \ref{thm:thm1}]
It will be convenient here to perform the ``one-shot coupling'' at time $t+3$
rather than at time $t+1$.  
By Corollary \ref{cor:corollary2}, $\mathbb{P}\left[u^{[t+3]C}\neq w^{[t+3]C}\right]$
is an upper bound for $d_{TV}\left(\mathcal{U}^{t+3},\mathcal{W}^{t+3}\right)$.
First, we restrict to the event $\left\{ J_{0}\leq\eta\right\}$.
Corollary \ref{cor:corollary10} tells us that
\[
\mathbb{E}\left[R_{t+2}-1\left|\mathscr{F}{}_{0}\right.\right]\;\leq\; 3r^{t/2d}\left(R_{0}-1\right) \,.
\]
Therefore by Lemma \ref{lem:lemma3},  Jensen's inequality, and the bound 
$1-(1+y)^{-p}\leq py$ for $p,y\geq 0$ (easily shown by calculus),
\begin{eqnarray}
\mathbb{P}\left[u^{[t+3]C}\neq w^{[t+3]C}\left|\mathscr{F}{}_{0}\right.\right] & = & 
  \mathbb{E}\left[\left. \mathbb{P}\left[u^{[t+3]C}\neq w^{t+3}\left|\mathscr{F}_{t+2}\right.\right]\right|
   \mathscr{F}_{0}\right]\nonumber \\
 & \leq & \mathbb{E}\left[1-\left(R_{t+2}\right)^{-\left(a_{2}+a_{3}+a_{4}+a_{5}\right)}\right]
  \nonumber \\
 & \leq & 1-\left(\mathbb{E}\left[R_{t+2}\right]\right)^{-\left(a_{2}+a_{3}+a_{4}+a_{5}\right)}
 \nonumber \\
 & \leq & 1-\left(1+3\,r^{t/2d}\left(R_{0}-1\right)\right)^{-\left(a_{2}+a_{3}+a_{4}+a_{5}\right)}
  \nonumber \\
  & \leq & 3\,r^{t/2d}\left(a_{2}+a_{3}+a_{4}+a_{5}\right)\left(R_{0}-1\right) \,.  \nonumber
\end{eqnarray}
This proves  the first statement of the theorem. 
If we no longer restrict to the event $\{J_{0}\leq\eta\}$,
then by Remark \ref{rem-Jeta}
(recall that $T=T(0)$ is the first time $t>0$ such that $J_{t}\leq\eta$),
\begin{eqnarray}
  \lefteqn{\mathbb{P}\left[\left. u^{[t+3]C}\neq w^{[t+3]C}\right|\mathscr{F}_{0}\right] }
    \nonumber \\
    & \leq & 
  \mathbb{P}\left[u^{[t+3]C}\neq w^{[t+3]C}\left|J_{0}>\eta,T\leq\left\lfloor \frac{t}{2}\right\rfloor +3\right.\right]+\mathbb{P}\left[\left.T>\left\lfloor \frac{t}{2}\right\rfloor +3\right|J_{0}>\eta\right]
  \nonumber \\
 & \leq & 3 \,r^{t/4d}\left(a_{2}+a_{3}+a_{4}+a_{5}\right)\left(R_{0}-1\right) \,+\,
  \frac{\max\left\{ J_{0},\eta\right\} \beta^{\left\lfloor \frac{t}{2}\right\rfloor +3}}{\eta} \,.
  \label{eq:unotv2}
\end{eqnarray}
Since this is greater than what we have on $\left\{ J_{0}\leq\eta\right\} $,
it is also a bound for general values of $J_{0}$. 
\end{proof}
\medskip{}

\section{\label{sub:Sampling}Sampling from equilibrium}

It is not hard to apply our previous results to obtain a bound on the rate of convergence
to the equilibrium distribution $\pi$ given by ($\ref{eq:equilibrumdensity}$).

\begin{proof}[Proof of Corollary \ref{cor:cor2}]
Fix $\mathcal{U}^{0}$ and let
$\mathcal{W}^{0}$ be a random vector with density $\pi$.
Define $u^t$ and $w^t$ accordingly.   By   (\ref{eq:unotv2}), we have
\[
  \mathbb{P}\left[\left. u^{[t+3]C}\neq w^{[t+3]C}\right|\mathcal{W}_{0}\right] \; \leq \; 
  3\,r^{t/4d}\left(a_{2}+a_{3}+a_{4}+a_{5}\right)\left(R_{0}-1\right)+
  \frac{\max\left\{ J_{0},\eta\right\} \beta^{\left\lfloor \frac{t}{2}\right\rfloor +3}}{\eta}\,.
\]
The corollary now follows from Corollary \ref{cor:corollary2}. 
\end{proof}
 
Now let $C_{g}:=\int\left(\prod_{i=1}^{4}z_{i}^{a_{i}+a_{i+1}-1}\right)\exp\left(\sum_{i=1}^{5}-z_{i}z_{i-1}\right)dz$.
Then we can bound the terms $\mathbb{E}_{\pi}\left[R_{0}\right]$
and $\mathbb{E}_{\pi}\left[J_{0}\right]$ in Corollary \ref{cor:cor2}
in the following way: 
\begin{eqnarray*}
d_{TV}\left(\mathcal{U}^{t+3},\pi\right)  
 & \leq & 3 \, r^{t/4d} \left(a_{2}+a_{3}+a_{4}+a_{5}\right) \,\frac{1}{C_{g}} \times  \\
 & & \int\left(\frac{\max\left\{ 1,v_2,v_4\right\} }{\min\left\{ 1,v_2,v_4\right\} }\right)
 \left(\prod_{i=1}^{4}v_{i}^{a_{i}+a_{i+1}-1}\right)\exp\left(\sum_{i=1}^{5}-v_{i}v_{i-1}\right)dv\\
 & & + \frac{\beta^{\left\lfloor \frac{t}{2}\right\rfloor +3}}{\eta} \left(\eta+\frac{1}{C_{g}}\int J_{0}\left(\prod_{i=1}^{4}v_{i}^{a_{i}+a_{i+1}-1}\right)\exp\left(\sum_{i=1}^{5}-v_{i}v_{i-1}\right)dv\right)\\
 & \leq & 3 \tilde{C_{\pi}}r^{t/4d }\left(a_{2}+a_{3}+a_{4}+a_{5}\right)+\left(\frac{\tilde{C_{J}}}{\eta}+1\right)\beta^{\left\lfloor \frac{t}{2}\right\rfloor +3}
\end{eqnarray*}
where $\tilde{C_{\pi}}:=\int\left(\frac{\max\{ 1,v_2,v_4\} }{\min\{ 1,v_2,v_4\} }\right)
\left(\prod_{i=1}^{4}v_{i}^{a_{i}+a_{i+1}-1}\right)\exp\left(\sum_{i=1}^{5}-v_{i}v_{i-1}\right)dv/C_{g}$
 and\\
 $\tilde{C_{J}}:=\int J_{0}\left(\prod_{i=1}^{4}v_{i}^{a_{i}+a_{i+1}-1}\right)\exp\left(\sum_{i=1}^{5}-v_{i}v_{i-1}\right)dv/C_g$.
We derive bounds for these terms in Appendix B in \cite{key-18}.

For the purpose of illustrating this result in a concrete example,
let us set $x=2$, $b=3$ and $a_{i}=i$. From Appendix B in  \cite{key-18} we get
$\tilde{C_{\pi}}\leq$ 31,065  $\tilde{C_{J}}\leq59$,
$\beta\leq7/9$, $r\leq1-\frac{3}{4356}$, $10\leq\eta\leq11$ and
$9\leq d\leq10$. Hence
\[
d_{TV}\left(\mathcal{U}^{t+3},\pi\right) \;\leq \;
31065*43\left(1-\frac{3}{4356}\right)^{\frac{t}{40}}+\left(1+\frac{59}{20}\right)\left(\frac{7}{9}\right)^{\left\lfloor \frac{t}{2}\right\rfloor +3}
\]
which implies that $d_{TV}\left(\mathcal{U}^{t+3},\pi\right)\leq10^{-5}$
for $t\geq$ 1,050,000.

\section{A brief look at the case $n=3$}

The case $n=3$ can be treated in a very similar manner as was used for $n=4$. 
The problem reduces to dealing with a Markov chain of a single variable, namely the 
second coordinate of the three,  given by 
\begin{equation}
u^{t+1}=\frac{\gamma_{2}^{t+1}}{\frac{\gamma_{1}^{t+1}}{u^{t}+x}+\frac{\gamma_{3}^{t+1}}{u^{t}+b}}\label{eq:s415MC}
\end{equation}
The uniform coupling of two chains $u^t$ and $w^t$  
with the property $u^{0}\leq w^{0}$ results in  $u^{t}\leq v^{t}$ for all $t$.
If $u^0<w^0$, then it is not hard to see that the ratio $R_{t}=\frac{w^{t}}{u^{t}}$
is strictly decreasing, hence we no longer need to define a process
like (\ref{eq:ratiomethod}) and we can simply work with this ratio directly.
Indeed, $R_{t+1}=R_{t}Q_{t}$ where 
\begin{eqnarray*}
Q_{t} & := & 1-\frac{\left(1-\frac{1}{R_{t}}\right)\left(x\gamma_{1}^{t+1}/\left(\left(u^{t}+x\right)\left(1+\frac{x}{w^{t}}\right)\right)+b\gamma_{3}^{t+1}/\left(\left(u^{t}+b\right)\left(1+\frac{b}{w^{t}}\right)\right)\right)}{\left(\gamma_{1}^{t+1}/\left(1+\frac{x}{w^{t}}\right)+\gamma_{3}^{t+1}/\left(1+\frac{b}{w^{t}}\right)\right)}\\
 & \leq & 1-\frac{\left(1-\frac{1}{R_{t}}\right)\left(x\gamma_{1}^{t+1}+b\gamma_{3}^{t+1}\right)}{\left(\gamma_{1}^{t+1}/\left(1+\frac{x}{w^{t}}\right)+\gamma_{3}^{t+1}/\left(1+\frac{b}{w^{t}}\right)\right)\left(u^{t}+\max\left\{ x,b\right\} \right)\left(1+\frac{\max\left\{ x,b\right\} }{u^{t}}\right)}\\
 & \leq & \rrr{r}_{t}+\frac{1-\rrr{r}_{t}}{R_{t}}
\end{eqnarray*}
where $\rrr{r}_{t}:=1-\min\left\{ x,b\right\} /\left(\left(u^{t}+\max\left\{ x,b\right\} \right)\left(1+\frac{\max\left\{ x,b\right\} }{u^{t}}\right)\right)$.
Note that if we define $K_{1,t+1}:=u^{t+1}$
and $K_{2,t+1}:=\frac{1}{u^{t+1}}$ then $K_{1,t+1}\leq\frac{\gamma_{2}^{t+1}}{\gamma_{1}^{t+1}+\gamma_{3}^{t+1}}\left(u^{t}+x+b\right)$ and $K_{2,t+1}\leq\left(\frac{\gamma_{1}^{t+1}}{\gamma_{2}^{t+1}}\frac{1}{x}+\frac{\gamma_{3}^{t+1}}{\gamma_{2}^{t+1}}\frac{1}{b}\right)$, and hence we do not need a process analogous to $D_{t}$ from the
previous section, since 
\[
\rrr{r}_{t+1} \;\leq\;  1-\min\left\{ x,b\right\} /\left(\left(K_{1,t+1}+\max\left\{ x,b\right\} \right)\left(1+\max\left\{ x,b\right\} K_{2,t+1}\right)\right)
\]
As before, we will require that $a_{1}+a_{4}>1$ in order that $\mathbb{E}\left[\gamma_{2}/\left(\gamma_{1}+\gamma_{3}\right)\right]<1$, and $a_{2}+a_{3}>1$ in order that $\mathbb{E}\left[\frac{\gamma_{1}^{t+1}}{\gamma_{2}^{t+1}}\right]<\infty$. If $J_{t}:=K_{1,t}+K_{2,t}$ and $S$ is a measurable stopping time such that $J_{S}\leq\eta$, with 
\[
\eta:=2\left(\frac{\left(x+b\right)\left(a_{2}+a_{3}\right)}{a_{1}+a_{2}+a_{3}+a_{4}-1} + \frac{\left(a_{1}+a_{2}\right)/x + \left(a_{3}+a_{4}\right)/b}{a_{2}+a_{3}-1}\right)/\left(1-\frac{a_{2}+a_{3}}{a_{1}+a_{2}+a_{3}+a_{4}-1}\right)
\] then we can repeat the steps of (\ref{eq:rs2}) 
\begin{eqnarray}
  \mathbb{E} \left[ R_{S + 1} \right] & = & \mathbb{E} \left[ Q_S R_S \right]
  \nonumber\\
  & \leq & \mathbb{E} \left[ \rrr{r}_S  \left( R_S - 1 \right) \right] + 1
  \nonumber\\
  & = & \mathbb{E} \left[ \left( 1 - \frac{\min \left\{ x, b \right\}}{\left(
  u^S + \max \left\{ x, b \right\} \right)  \left( 1 + \frac{\max \left\{ x, b
  \right\}}{u^S} \right)} \right)  \left( R_S - 1 \right) \right] + 1
  \nonumber\\
  & \leq & \mathbb{E} \left[ \left( 1 - \frac{\min \left\{ x, b
  \right\}}{\left( u^S + 2 \max \left\{ x, b \right\} + \max \left\{ x, b
  \right\}^2 / u^S \right) } \right)  \left( R_S - 1 \right) \right]
  \nonumber\\
  & \leq & \mathbb{E} \left[ \left( 1 - \frac{\min \left\{ x, b \right\}}{
  \left( 2 \max \left\{ x, b \right\} + \eta \left( 1 + \max \left\{ x, b
  \right\}^2 \right) \right)} \right)  \left( R_S - 1 \right) \right] + 1
  \nonumber\\
  & = & r \, \mathbb{E}  \left[ R_S - 1 \right] + 1  \label{eq:n3S}
\end{eqnarray}
where 
$r = 1 - \min \left\{ x, b \right\} / \left(  2 \max \{ x, b\} + \eta \left( 1 + \max \{ x, b\}^2 \right) \right)$.
Note that we no longer need to look at time $S+2$ in the left-hand
side of (\ref{eq:n3S}) in order to obtain this inequality. This means
that from the proof of Lemma \ref{lem:lemma9} and Corollary \ref{cor:corollary10}
we get 
\[
\mathbb{E}\left[R_{t+1}\left|J_{0}\leq\eta\right.\right]\leq1+3r^{\frac{t}{d}}\left(R_{0}-1\right)
\]
where $d\,=\,\max\left\{3,\ln(\beta|\ln \beta|r/2)/\ln \beta\right\}$.
From the proof of Theorem \ref{thm:thm1} we conclude
\begin{thm}
\label{thm:s4n3} $\left[n=3\right]$ Suppose that $a_{1}+a_{4}>1$ and $a_{2}+a_{3}>1$.
If $u^{t}$ and $w^{t}$ are two instances of the Markov chain (\ref{eq:s415MC}), then
\[
d_{TV}\left(u^{t+2},w^{t+2}\right) \;\leq\; r^{^{\frac{t}{2d}}}\left(1+3\left(a_{2}+a_{3}\right)\left(R_{0}-1\right)\right)+\frac{\max\left\{ J_{0},\eta\right\} \beta^{\left\lfloor \frac{t}{2}\right\rfloor +3}}{\eta} \,.
\]
\end{thm}
We can make an analogous argument to obtain a result similar to Corollary
\ref{cor:cor2}. In particular if we let $\mathcal{U}^{0}=\left(1,1,1\right)$,
$\mathcal{W}^{0}\sim\pi$ and $x=1$, $b=2$ and $a_{i}=i$, then
by calculations similar to those done in Section \ref{sub:Sampling}
we get 

\[
d_{TV}\left(\mathcal{U}^{t+2},\pi\right)\leq600\left(1-\frac{78}{79}\right)^{\frac{t}{20}}+6\left(\frac{7}{9}\right)^{\left\lfloor \frac{t}{2}\right\rfloor +3}
\]
which in particular implies that $d_{TV}\left(\mathcal{U}^{t+2},\pi\right)\leq10^{-5}$
for $t\geq$ 14,000.
\section*{Appendix A}

\begin{tabular}{|c|}
\hline 
$C_{1}=\frac{a_{2}+a_{3}}{a_{1}+a_{2}+a_{3}+a_{4}-1}x+\frac{a_{4}+a_{5}}{b}$, \hspace{10mm}
$C_{2}=\frac{a_{1}+a_{2}+xb}{x\left(a_{2}+a_{3}+a_{4}+a_{5}-1\right)}$\tabularnewline
\hline 
$\varrho=\frac{4\left(a_{3}+a_{4}\right)}{\left(a_{1}+a_{2}-\frac{1}{3}\right)}$\tabularnewline
\hline  
$\eta=\frac{C_{1}+C_{2}}{1-\max\left\{ \left(a_{2}+a_{3}\right)/\left(a_{1}+a_{2}+a_{3}+a_{4}-1\right),\left(a_{3}+a_{4}\right)/\left(a_{2}+a_{3}+a_{4}+a_{5}-1\right)\right\} }$\tabularnewline
\hline  
$\theta_{1}=\frac{1}{x}\left(\varrho+4\right)\frac{a_{2}+a_{3}}{a_{1}+a_{2}+a_{3}+a_{4}-1}$\tabularnewline
\hline 
$\theta_{2}=\mathbb{E}\left[\left(2+\frac{\gamma_{2}}{\gamma_{4}}+\frac{\gamma_{4}}{\gamma_{2}}\right)\left(\frac{\gamma_{3}}{\gamma_{2}+\gamma_{4}}\right)\right]\left(\left(\varrho+4\right)x+\frac{4\left(a_{3}+a_{4}\right)}{b}\right)$\tabularnewline 
\hline 
$\theta_{3}=\frac{1}{x}\left(\varrho+4\right)\left(\frac{a_{2}+a_{3}}{a_{1}+a_{2}+a_{3}+a_{4}-1}x+\frac{a_{4}+a_{5}}{b}\right)+\left(\left(\varrho+4\right)x+\frac{a_{4}+a_{5}}{b}\right)\mathbb{E}\left[\left(2+\frac{\gamma_{2}}{\gamma_{4}}+\frac{\gamma_{4}}{\gamma_{2}}\right)\left(\frac{\frac{\gamma_{1}}{x}+b}{\gamma_{2}+\gamma_{4}}\right)\right]$\tabularnewline 
\hline 
$r=1-\left(\eta\left(\theta_{1}+\theta_{2}\right)+\theta_{3}\right)^{-1}$\tabularnewline 
\hline 
$\beta=\frac{1+\max\left\{ \left(a_{2}+a_{3}\right)/\left(a_{1}+a_{2}+a_{3}+a_{4}-1\right),\left(a_{3}+a_{4}\right)/\left(a_{2}+a_{3}+a_{4}+a_{5}-1\right)\right\} }{2}$\tabularnewline 
\hline 
$d\,=\,\max\left\{3,\ln(\beta|\ln \beta|\sqrt{r}/2)/\ln \beta\right\}$\tabularnewline
\hline 
\end{tabular}

\medskip
We can calculate $\theta_2$ and $\theta_3$ with the help of partial fractions, as follows.  
Writing $A_i=a_i+a_{i+1}$, we obtain
\[
   \mathbb{E}\left(  \left(\frac{\gamma_2}{\gamma_4}+\frac{\gamma_4}{\gamma_2}\right)\,
       \frac{1}{\gamma_2+\gamma_4}\right)  \;=\;
      \mathbb{E}\left(  \frac{1}{\gamma_2}+\frac{1}{\gamma_4}-\frac{2}{\gamma_2+\gamma_4}\right)    \;=\;
    \frac{A^2_2+A_4^2-A_2-A_4}{(A_2-1)(A_4-1)(A_2+A_4-1)}\,.
\]

\medskip

\section*{Acknowledgement}

The research of Neal Madras has been supported in part by a Discovery
Grant from the Natural Science and Engineering Research Council of
Canada.

\end{document}